\documentclass[12pt]{amsart}
\usepackage{amsfonts,amssymb,amsmath,amscd,amsthm,amstext,bm}
\usepackage{xcolor}
\usepackage[normalem]{ulem}
\usepackage{enumerate}
\usepackage{fancyhdr}
\usepackage{geometry}
\usepackage{graphicx}
\usepackage{epstopdf}
\usepackage{mathrsfs}
\usepackage{txfonts}
\usepackage{tikz}
\usetikzlibrary{arrows.meta, positioning, decorations.pathreplacing}
\usepackage{url}
\usepackage{comment}
\usepackage[colorlinks, citecolor=blue, linkcolor=blue, urlcolor=blue, pagebackref, hypertexnames=false]{hyperref}
\allowdisplaybreaks
\numberwithin{equation}{section}
\newtheorem{theorem}{Theorem}[section]
\newtheorem{proposition}[theorem]{Proposition}
\newtheorem{lemma}[theorem]{Lemma}
\newtheorem{corollary}[theorem]{Corollary}
\theoremstyle{definition}
\newtheorem{definition}[theorem]{Definition}
\newtheorem{example}[theorem]{Example}
\theoremstyle{remark}
\newtheorem{remark}[theorem]{Remark}

\newcommand{\bbR}{\mathbb R}
\newcommand{\bbC}{\mathbb C}

\newcommand{\bbZ}{\mathbb Z}

\newcommand{\calC}{\mathcal C}
\newcommand{\calD}{\mathcal D}

\newcommand{\calH}{\mathcal H}
\newcommand{\calK}{\mathcal K}
\newcommand{\calL}{\mathcal L}
\newcommand{\calM}{\mathcal M}

\newcommand{\calS}{\mathcal S}

\newcommand{\calW}{\mathcal W}

\newcommand{\supp}{\operatorname{supp}}

\newcommand{\Fkap}{\mathcal F_{\kappa}}

\newcommand{\CH}{\operatorname{CH}}

\newcommand{\norm}[1]{\left\|#1\right\|}

\newcommand{\ip}[2]{\left\langle #1,#2\right\rangle}

\renewcommand{\epsilon}{\varepsilon}

\title[Chamber lifting and non-radial Dunkl multipliers]
{Chamber lifting and non-radial Dunkl multipliers\\[15pt] {\rm \tiny In Honor of Ngai-Ching Wong}\\[15pt] }

\author{Der-Chen Chang}
 \address{Der-Chen Chang, Department of Mathematics and Statistics, Georgetown University, Washington, DC 20057, USA; \,
Graduate Institute of Business Administration, College of Management, Fu Jen Catholic University, New Taipei City 242,
Taiwan, ROC}
\email{chang@georgetown.edu}

\author{Ji Li}
\address{Ji Li, School of Mathematical and Physical Sciences, Macquarie University, NSW 2109, Australia}
\email{ji.li@mq.edu.au}

\author{Chaojie Wen\textsuperscript{*}}
\address{Chaojie Wen, Department of Mathematics, Sun Yat-sen University, Guangzhou, 510275, P.R.~China}
\email{wenchj@mail2.sysu.edu.cn}
\thanks{\textsuperscript{*}Corresponding author.}

\author{Liangchuan Wu}
\address{ Liangchuan Wu, School of Mathematical Science, Anhui University, Hefei, 230601,  P.R.~China}
\email{wuliangchuan@ahu.edu.cn}



\begin{document}

\begin{abstract}
We study non-radial Dunkl multipliers via chamber lifting. For an arbitrary finite reflection group $G$, the chamber lifting
records all reflected values of a function and conjugates a multiplier into a finite matrix-valued operator on the chamber.
If the dyadic matrix entries admit off-diagonal kernels satisfying the chamber $L^2$ H\"ormander condition $\CH^2_{s,\eta}$
with $s>N_\kappa/2$, then the original multiplier is bounded on $L^p(\bbR^N,d\omega)$ for every $1<p<\infty$.

For the product reflection group $\Sigma_N=A_1^N\simeq\bbZ_2^N$ this chamber condition follows from scalar Sobolev conditions on the Walsh pieces of the multiplier. 
The tensor product of the one-dimensional even/odd Dunkl decompositions, together with the finite Walsh transform, identifies each lifted matrix entry with a Hankel multiplier acting between parity components. 
Wall separation and a scale-invariant $L^2$ Sobolev condition of order $\sigma>N_\kappa/2$ therefore imply $L^p$ boundedness, for
all $1<p<\infty$, for a genuinely non-radial class of symbols. 
The order $N_\kappa/2$ is forced already by the rank-one Bessel transform.

The same chamber theorem also applies to non-product examples once the matrix kernel condition is known, including the dihedral groups $I_2(q)$ and hence $A_2\simeq I_2(3)$ and $B_2\simeq I_2(4)$. 
The scalar Walsh--Sobolev verification is specific to $A_1^N$. In non-product groups such as $A_2$, $A_{N-1}$, and $B_N$, the product parity calculus is absent, so a scalar
theorem of the same form would require additional transform estimates.
\end{abstract}

\maketitle


\section{Introduction}
\label{s:1}

\subsection{Non-radial multipliers and the chamber reduction}

Let $R$ be a finite root system in $\bbR^N$, let $G$ be the finite reflection group generated by $R$, and let $\kappa\ge0$ be a $G$-invariant multiplicity function; 
these are the standard data for the rational Dunkl transform associated with Dunkl operators \cite{Dunkl,Dunkl1991,deJeu1993,O1993,Rosler2003,Anker2017}. We write $d\omega$ for the associated Dunkl measure, namely
$$
    d\omega(x)=h_\kappa(x)^2\,dx, \quad   N_\kappa=N+2\gamma_\kappa, \quad   \text{and }\quad  \gamma_\kappa=\sum_{\alpha\in R_+}\kappa(\alpha).
$$
For general background on the Dunkl transform, the intertwining operator, the Dunkl kernel, and related
orthogonal expansions, see also \cite{DX2014,R1998,R1999,RJ2002}. For $m\in L^\infty(\bbR^N)$, the Dunkl multiplier $T_m$ is defined on $L^2(d\omega)$ by $\Fkap(T_mf)(\xi)=m(\xi)\Fkap f(\xi)$. 
Our aim is to prove, from explicit conditions on a non-radial symbol $m$, estimates of the form
$$
    \|T_m f\|_{L^p(\bbR^N,d\omega)}\le C_p\|f\|_{L^p(\bbR^N,d\omega)},\quad  1<p<\infty.
$$
The natural $L^2$ H\"ormander scale is $s>N_\kappa/2$. The usual argument fails at the kernel step: for non-radial kernels
there is no translation formula strong enough to run the Calder\'on--Zygmund proof directly on the full space. This explains
why several existing multiplier theorems impose radiality, work with radial functions, or use semigroup representations. The translation and convolution theory in the Dunkl setting is developed in \cite{T2002,TX2005}, while
positive radial translation phenomena are closely related to \cite{R2003Radial}. Here we avoid estimates for general
Dunkl translations. Instead we pass to a fundamental chamber and retain all reflected values of the function as components
of a finite vector.

\subsection{The chamber lifting}

Let $G$ be a finite reflection group and let $\calC$ be an open fundamental chamber. 
We use the chamber lifting developed for Dunkl Calder\'on-type commutators in \cite{HLLSW-Calderon}:
$$
    Uf(x)=\bigl(f(\sigma x)\bigr)_{\sigma\in G} \quad \text{for}\quad x\in\calC.
$$
Up to the harmless finite counting measure, $U$ is an isometric identification of $L^p(\bbR^N,d\omega)$ with $L^p(\calC,d\omega;\bbC^G)$. Thus a scalar multiplier becomes
$$
    \mathbb T_m^G=UT_mU^{-1},
$$
a finite matrix operator on one chamber. The chamber lift imposes no symmetry on $f$; it writes the full-space question as a vector-valued singular integral estimate on $\calC$.

This gives the first main theorem. The notation $\CH^2_{s,\eta}(\calC)$ denotes the chamber $L^2$ H\"ormander condition obtained from Definition \ref{d:3} by replacing the product chamber by $\calC$, the product index set by $G$, and the chamber volume by $V_{\calC}(y,r)=\omega(B(y,r)\cap\calC).$

\begin{theorem} \label{t:1} 
Let $G$ be a finite reflection group on $\bbR^N$, let $m\in L^\infty(\bbR^N)$, and let $\mathbb T_m^G=UT_mU^{-1}$. 
Assume that the dyadic entries of $\mathbb T_m^G$ admit off-diagonal kernels and that $\mathbb T_m^G\in\CH^2_{s,\eta}(\calC)$ for some $s>N_\kappa/2$ and $0<\eta\le1$. 
Then, for every $1<p<\infty$,
$$
    \|T_mf\|_{L^p(\bbR^N,d\omega)}  \le C_{p,G,\kappa,s,\eta}  \bigl(\|m\|_\infty+\|\mathbb T_m^G\|_{\CH^2_{s,\eta}(\calC)}\bigr)   \|f\|_{L^p(\bbR^N,d\omega)}.
$$
No $G$-invariance, radiality, or chamber-constancy condition is imposed on $f$.
\end{theorem}

This chamber theorem is formulated in terms of matrix kernels. 
It is not, by itself, a scalar smoothness criterion: for a general reflection group there is no product parity calculus which turns a scalar condition on $m$ directly into the matrix estimates above. The scalar verification is an additional product-group argument.

\subsection{The product scalar theorem}

In the product case $\Sigma_N=A_1^N\simeq\bbZ_2^N$ and $\calC=(0,\infty)^N$, the lift is $Uf(x)=\bigl(f(\varepsilon x)\bigr)_{\varepsilon\in\{\pm1\}^N}$. 
The one-dimensional even/odd Dunkl decomposition tensorizes; this is the product version of the Bessel and
$h$-harmonic structure discussed, for instance, in \cite{Rosler2003,DX2014}. After passing to Walsh parity coordinates, 
the matrix entries are expressed through the Walsh pieces of the scalar symbol. This transform-side identity has no analogue of the same form for most non-product groups.

\begin{theorem} \label{t:2} 
Let $\Sigma_N\simeq\bbZ_2^N$, let $m\in L^\infty(\bbR^N)$, and let $\mathbb T_m=\calW U T_m U^{-1}\calW^{-1}$ be the lifted multiplier in Walsh parity coordinates. 
Assume that the dyadic entries of $\mathbb T_m$ admit off-diagonal kernels and that $\mathbb T_m\in\CH^2_{s,\eta}$ for some $s>N_\kappa/2$ and $0<\eta\le1$. 
Then, for every $1<p<\infty$,
$$
    \|T_m f\|_{L^p(\bbR^N,d\omega)}   \le C_{p,N,\kappa,s,\eta}  \left(\|m\|_\infty+\|\mathbb T_m\|_{\CH^2_{s,\eta}}\right)   \|f\|_{L^p(\bbR^N,d\omega)}.
$$
\end{theorem}

For $R\in\calC$ and $\theta\in\{0,1\}^N$, define $m_\theta(R)=2^{-N}\sum_{\varepsilon\in\{\pm1\}^N}\varepsilon^\theta m(\varepsilon R)$. 
Fix $0<c_0<N^{-1/2}$ and set
$$
    \Gamma_{c_0}=\{R\in\calC:R_i\ge c_0|R|\hbox{ for every }i\}.
$$
Support conditions for $L^\infty$ symbols are always understood in the essential-support sense. 
Let $\psi\in C_c^\infty((1/2,2))$ be non-negative and satisfy $\sum_{j\in\bbZ}\psi(2^{-j}r)=1$ for $r>0$. 
Put
$$
    \psi_j(r)=\psi(2^{-j}r)\quad \text{and}\quad  a_{\theta,j}(R)=\psi_j(|R|)m_\theta(R).
$$
Extending $a_{\theta,j}(2^j\cdot)$ by zero outside $\calC$, define
$$
    \calS_\sigma(m)=  \max_{\theta\in\{0,1\}^N}  \sup_{j\in\bbZ}   \left\|a_{\theta,j}(2^j\cdot)\right\|_{W_2^\sigma(\bbR^N)}.
$$

\begin{theorem}
\label{t:3} 
Let $\Sigma_N\simeq\bbZ_2^N$, let $m\in L^\infty(\bbR^N)$, 
and suppose that $\operatorname*{ess\,supp}m_\theta\subset\Gamma_{c_0}$ for every $\theta\in\{0,1\}^N$. 
Assume that $\calS_\sigma(m)<\infty$ for some $\sigma>N_\kappa/2$. Then $T_m$ extends boundedly to $L^p(\bbR^N,d\omega)$ for every $1<p<\infty$. More precisely,
$$
    \|T_mf\|_{L^p(\bbR^N,d\omega)} \le C_{p,N,\kappa,\sigma,c_0,\psi}
        \left(\|m\|_\infty+\calS_\sigma(m)\right)   \|f\|_{L^p(\bbR^N,d\omega)}.
$$
\end{theorem}

\begin{remark}
The condition $\sigma>N_\kappa/2$ is the standard open H\"ormander range at the sharp scale. 
On the chamber side, the passage from $\CH^2_{s,\eta}$ to the integral H\"ormander condition uses
$$
    \int_{\calC}\left(1+\frac{\rho(x,y)}{r}\right)^{-2s}d\omega(x)   \lesssim V(y,r),
$$
and this argument requires $2s>N_\kappa$. 
On the multiplier side, rank one already prevents lowering the scale: for $\Sigma_1=A_1\simeq\bbZ_2$ the even part is the Hankel transform for the Bessel measure $x^{2\kappa}\,dx$, 
whose homogeneous dimension is $1+2\kappa=N_\kappa$, and imaginary powers in the sharp Bessel multiplier theorem rule out Sobolev order below
$N_\kappa/2$ \cite{KP2019}. This is the range proved here and the one used throughout the paper.
\end{remark}

For an integer $M\ge0$, let
$$
    \calM_M(m):=   \max_{\theta\in\{0,1\}^N}
        \max_{|\nu|\le M} \sup_{R\in\calC} |R|^{|\nu|}\left|\partial_R^\nu m_\theta(R)\right|.
$$

\begin{corollary}
\label{c:1} 
Let $\Sigma_N\simeq\bbZ_2^N$, let $m\in L^\infty(\bbR^N)$, and suppose that $\calM_M(m)<\infty$ and 
$\operatorname*{ess\,supp}m_\theta\subset\Gamma_{c_0}$ for every $\theta\in\{0,1\}^N$, for some integer $M>N_\kappa/2$. 
Then $T_m$ extends boundedly to $L^p(\bbR^N,d\omega)$ for every $1<p<\infty$, with
$$
    \|T_mf\|_{L^p(\bbR^N,d\omega)}   \le C_{p,N,\kappa,M,c_0,\psi}    \left(\|m\|_\infty+\calM_M(m)\right)   \|f\|_{L^p(\bbR^N,d\omega)}.
$$
\end{corollary}

\subsection{Positive examples and limitations} \label{s:2}

The product scalar theorem contains non-radial symbols. If $m(\xi)=a(\xi/|\xi|)$ with $a\in C^\infty(\mathbb S^{N-1})$, 
and the angular support of $a$ is contained in $\{|\xi_i|\ge c_0|\xi|\hbox{ for all }i\}$, 
then the conditions of Corollary \ref{c:1} hold. Unless $a$ is constant on the sphere, this is not a radial spectral multiplier.

The abstract theorem gives further reflection-group examples under the chamber matrix kernel condition. 
For the dihedral group $I_2(q)$, the fundamental chamber is a wedge of angle $\pi/q$. 
If the lifted $2q\times2q$ chamber kernels satisfy $\CH^2_{s,\eta}$, Theorem \ref{t:1} gives full $L^p$ boundedness. 
Thus $A_2\simeq I_2(3)$ and $B_2\simeq I_2(4)$ are specific non-product cases. 
The following rank-two statement is just Theorem \ref{t:1} written with the dihedral notation used later.

\begin{theorem} \label{t:4} 
Let $G_q=I_2(q)$ be the dihedral reflection group of order $2q$ on $\bbR^2$, let $\calC$ be a fundamental chamber, and set $U_qf(x)=\bigl(f(\sigma x)\bigr)_{\sigma\in G_q}$, $x\in\calC$. 
For $m\in L^\infty(\bbR^2)$, set $\mathbb T_m^{(q)}=U_qT_mU_q^{-1}$. 
Assume that the dyadic entries of $\mathbb T_m^{(q)}$ admit off-diagonal kernels and that $\mathbb T_m^{(q)}\in\CH^2_{s,\eta}(\calC)$ for some $s>N_\kappa/2$ and $0<\eta\le1$. Then, for every $1<p<\infty$,
$$
     \|T_mf\|_{L^p(\bbR^2,d\omega)}  \le C_{p,q,\kappa,s,\eta}
        \bigl(\|m\|_\infty+\|\mathbb T_m^{(q)}\|_{\CH^2_{s,\eta}(\calC)}\bigr)   \|f\|_{L^p(\bbR^2,d\omega)}.
$$
No $G_q$-invariance or radiality condition is imposed on $f$.
\end{theorem}

For $A_2\simeq I_2(3)$ the chamber has angle $\pi/3$; for $B_2\simeq I_2(4)$ it may be written as $\{(x_1,x_2):x_1>x_2>0\}$. These are the two explicit non-product examples used below.

The distinction between abstract and scalar theorems is essential. 
The group $A_1^N$ has only one-dimensional Walsh characters and the Dunkl kernel is a tensor product of one-dimensional even/odd Bessel kernels. 
For $A_2\simeq S_3$, the Weyl group has a two-dimensional irreducible representation, so there is no analogous Walsh diagonalization of the chamber matrix.
For $B_N$, $N\ge3$, the Coxeter boundary includes walls of the form $x_i=x_j$ in addition to coordinate walls, whenever the corresponding root multiplicity is present; 
a condition separating only from coordinate walls then does not see all boundary faces. 
For $A_{N-1}$, the chamber $x_1>\cdots>x_N$ has no sign-parity calculus. 
These observations are not counterexamples to multiplier boundedness. They show that the scalar Walsh--Sobolev proof does not formally extend to these groups.

\subsection{A pointwise criterion and spectral diagonal examples}

We also record two consequences of the chamber theorem. 
First, a pointwise dyadic kernel condition with decay order $M>N_\kappa+1$ implies 
the chamber $\CH^2_{s,\eta}$ condition for some $s>N_\kappa/2$ and hence gives strong $L^p$ bounds; 
a balanced pointwise variant with denominator $(V(x,r)V(y,r))^{1/2}$ is also included because this form occurs naturally in heat-kernel estimates. 
If, in addition, the principal-value truncations are uniformly $L^2$ bounded or converge in $L^2$, the pointwise criterion gives weak type $(1,1)$.

\begin{corollary}\label{c:2} 
Let $M>N_\kappa+1$. Suppose that the dyadic chamber kernels of the entries of $\mathbb T_m$ satisfy 
the pointwise size and first difference estimates of Definition \ref{d:5} with order $M$ and some $0<\eta\le1$. 
Then the conclusion of Theorem \ref{t:2} holds. 
If the associated truncated kernels also satisfy the admissible-truncation condition in Definition \ref{d:4}, then $T_m$ is of weak type $(1,1)$.
\end{corollary}

Second, Laplace-transform type spectral multipliers form a diagonal class in the chamber representation. 
For $m(\xi)=|\xi|^2\int_0^\infty e^{-t|\xi|^2}\phi(t)\,dt$, $\phi\in L^\infty(0,\infty)$, 
the theorem of Hassani--Sifi gives strong $L^p$ and weak type $(1,1)$ bounds in the $\bbZ_2^N$ Dunkl setting. 
We quote this result as an external spectral theorem and place it in the diagonal chamber formulation.

\begin{theorem} \label{t:5} 
Let $\Sigma_N\simeq\bbZ_2^N$, and let $m(\xi)=|\xi|^2\int_0^\infty e^{-t|\xi|^2}\phi(t)\,dt$, where $\phi\in L^\infty(0,\infty)$. Then
$$
    \|T_m f\|_{L^p(\bbR^N,d\omega)}    \le C_{p,N,\kappa}\|\phi\|_{\infty}   \|f\|_{L^p(\bbR^N,d\omega)}, \qquad 1<p<\infty.
$$
Moreover,
$$
    \omega\left\{x:|T_mf(x)|>\lambda\right\}   \le \frac{C_{N,\kappa}\|\phi\|_{\infty}}{\lambda}    \|f\|_{L^1(\bbR^N,d\omega)}.
$$
\end{theorem}

\subsection{Logical order of the main results}

The statements above are ordered by logical dependence rather than by generality of the symbol class. 
The abstract chamber theorem is the base result: once a lifted matrix kernel satisfies $\CH^2_{s,\eta}$, 
Calder\'on--Zygmund theory on the chamber gives the full-space $L^p$ estimate. 
The product scalar theorem supplies a verification of this condition: in the group $A_1^N$, the Walsh--Bessel decomposition shows that wall-separated Sobolev control of the scalar pieces $m_\theta$ implies the chamber estimates. 
The pointwise criterion gives a second way to verify the same estimates from dyadic kernel bounds. 
The Hassani--Sifi theorem is quoted as an external spectral multiplier theorem and then written in chamber coordinates. The implications used in the paper are
$$
\begin{array}{c}
\text{scalar Walsh--Sobolev control in }A_1^N\\
\Downarrow\\
\text{dyadic Bessel estimates between parity components}\\
\Downarrow\\
\mathbb T_m\in\CH^2_{s,\eta}\\
\Downarrow\\
\text{finite matrix Calder\'on--Zygmund theorem on one chamber}\\
\Downarrow\\
T_m:L^p(\bbR^N,d\omega)\to L^p(\bbR^N,d\omega).
\end{array}
$$
For a general finite reflection group, only the last two arrows are asserted in this paper. The first two arrows are
precisely the product-specific part.

\subsection{Organization}

The paper is organized into six sections after the introduction. 
Section \ref{s:3} develops the product chamber notation: the $A_1^N$ Dunkl setting, 
the chamber value lift, Walsh parity, and the product Hankel representation of the lifted multiplier. 
Section \ref{s:6} proves the chamber $L^2$ H\"ormander estimates and the finite matrix Calder\'on--Zygmund theorem, 
and then proves the abstract product chamber theorem. 
Section \ref{s:9} verifies the chamber condition from wall-separated Walsh--Sobolev control and proves Theorem \ref{t:3}.
Section \ref{s:10} contains the pointwise criterion and the finite-reflection-group chamber theorem, with the dihedral cases $A_2$ and $B_2$ written explicitly. 
Section \ref{s:13} records the Hassani--Sifi spectral theorem in chamber setting. 
Section \ref{s:14} compares the result with earlier multiplier theorems and lists the remaining scalar verification questions.

\section{Product chamber notation}
\label{s:3}

\subsection{\texorpdfstring{The group $\bbZ_2^N$}{The group z2d}}

Throughout the main part of the paper we work with the root system $R=\{\pm e_1,\ldots,\pm e_N\}$. 
The reflection group is generated by the coordinate reflections $\sigma_j(x_1,\ldots,x_j,\ldots,x_N)=(x_1,\ldots,-x_j,\ldots,x_N)$, 
hence $\Sigma_N\simeq \bbZ_2^N$. The multiplicity is an $N$-tuple $\kappa=(\kappa_1,\ldots,\kappa_N)$ with $\kappa_j\ge0$. The weight is
\begin{equation}\label{e:1}
        h_\kappa(x)=\prod_{j=1}^N |x_j|^{\kappa_j},    \qquad    d\omega(x)=\prod_{j=1}^N |x_j|^{2\kappa_j}\,dx.
\end{equation}
The number $\gamma_\kappa=\sum_{j=1}^N\kappa_j$ is the degree of homogeneity of $h_\kappa$, 
and therefore $d\omega(r x)=r^{N_\kappa}\,d\omega(x)$ with $N_\kappa=N+2\gamma_\kappa$.

The Dunkl operators are
\begin{equation}\label{e:2}
        D_j f(x) =\partial_j f(x)     +\kappa_j \frac{f(x)-f(\sigma_jx)}{x_j}, \qquad j=1,\ldots,N.
\end{equation}
The Dunkl Laplacian is $\Delta_\kappa=\sum_{j=1}^N D_j^2$. The Dunkl transform is
\begin{equation}\label{e:3}
        \Fkap f(\xi) =c_\kappa\int_{\bbR^N} f(x)E_\kappa(-i x,\xi)\,d\omega(x),
\end{equation}
where $E_\kappa$ is the product Dunkl kernel and $c_\kappa$ is the normalizing constant which makes $\Fkap$ unitary on $L^2(d\omega)$.

The exact value of $c_\kappa$ will never matter. We keep the same convention throughout the paper. All constants may depend on this convention.

\subsection{The chamber and its measure}

Let $\calC=(0,\infty)^N$. On $\calC$ we use the Euclidean distance $\rho(x,y)=|x-y|$. 
The chamber measure is the restriction $d\omega(x)=\prod_{j=1}^N x_j^{2\kappa_j}\,dx$, 
and for $x\in\calC$, $r>0$, we write $B(x,r)=\{y\in\calC:|x-y|<r\}$ and $V(x,r)=\omega(B(x,r))$.

\begin{lemma} \label{l:1} 
There are constants $C_1,C_2>0$, depending only on $N$ and $\kappa$, such that for all $x\in\calC$ and $r>0$,
\begin{equation}\label{e:4}
        C_1 r^N\prod_{j=1}^N (x_j+r)^{2\kappa_j}  \le V(x,r) \le  C_2 r^N\prod_{j=1}^N (x_j+r)^{2\kappa_j}.
\end{equation}
Consequently $(\calC,\rho,d\omega)$ is a space of homogeneous type and
\begin{equation}\label{e:5}
        V(x,R) \le C\left(\frac{R}{r}\right)^{N_\kappa}V(x,r), \qquad 0<r\le R.
\end{equation}
\end{lemma}

\begin{proof}
The ball $B(x,r)$ contains a rectangular box and is contained in a larger rectangular box with side lengths comparable to $r$. 
More precisely, for fixed dimensional constants,
$$
     \prod_{j=1}^N (x_j-c r,x_j+c r)\cap(0,\infty)^N    \subset B(x,r) \subset \prod_{j=1}^N (x_j-r,x_j+r)\cap(0,\infty)^N.
$$
It is therefore enough to estimate the one-dimensional integrals
$$
        \int_{\max(0,x_j-r)}^{x_j+r} t^{2\kappa_j}\,dt.
$$

For $a\ge0$ and $r>0$,
$$
   \int_{\max(0,a-r)}^{a+r} t^{2\kappa_j}\,dt \simeq r(a+r)^{2\kappa_j}.
$$
Indeed, if $a\le2r$, both sides are comparable to $r^{2\kappa_j+1}$. 
If $a>2r$, then $t\simeq a$ on the interval and the integral is comparable to $r a^{2\kappa_j}$, which is the same as $r(a+r)^{2\kappa_j}$. 
Multiplying these one-dimensional estimates proves \eqref{e:4}.

For \eqref{e:5}, use \eqref{e:4} twice:
$$
    \frac{V(x,R)}{V(x,r)}
     \lesssim \left(\frac{R}{r}\right)^N \prod_{j=1}^N\left(\frac{x_j+R}{x_j+r}\right)^{2\kappa_j}
     \le \left(\frac{R}{r}\right)^{N+2\sum_j\kappa_j}.
$$
This is \eqref{e:5}.
\end{proof}

\subsection{Weighted annular integrals}

The following estimate is used repeatedly. It is the place where the condition $s>N_\kappa/2$ enters.

\begin{lemma} \label{l:2} 
Let $s>N_\kappa/2$. Then for every $x\in\calC$, every $r>0$, and every $a\ge1$,
\begin{equation}\label{e:6}
        \int_{\{y:\rho(x,y)>ar\}} \left(1+\frac{\rho(x,y)}{r}\right)^{-2s}\,d\omega(y)  \le  C V(x,r)a^{N_\kappa-2s}.
\end{equation}
In particular,
\begin{equation}\label{e:7}
        \int_{\calC} \left(1+\frac{\rho(x,y)}{r}\right)^{-2s}\,d\omega(y) \le C V(x,r).
\end{equation}
\end{lemma}

\begin{proof}
Decompose the region $\rho(x,y)>ar$ into annuli
$$
     A_k=\{y:2^k ar<\rho(x,y)\le 2^{k+1}ar\},  \qquad k=0,1,2,\ldots.
$$
On $A_k$,
$$
        \left(1+\frac{\rho(x,y)}{r}\right)^{-2s} \le C(2^k a)^{-2s}.
$$
By the doubling condition,
$$
      \omega(A_k) \le V(x,2^{k+1}ar) \le C(2^k a)^{N_\kappa}V(x,r).
$$
Therefore
$$
      \int_{A_k}\left(1+\frac{\rho(x,y)}{r}\right)^{-2s}\,d\omega(y)  \le C(2^k a)^{N_\kappa-2s}V(x,r).
$$
Since $2s>N_\kappa$, the geometric series in $k$ converges and gives \eqref{e:6}. 
Taking $a=1$, and adding the ball $B(x,r)$, gives \eqref{e:7}.
\end{proof}

\subsection{Chamber lifting and parity} \label{s:4}

We first record the product version of the chamber value lift from \cite{HLLSW-Calderon}.

Let $\Sigma_N=\{\pm1\}^N.$ For $\varepsilon\in\Sigma_N$ and $x\in\calC$, put $\varepsilon x=(\varepsilon_1x_1,\ldots,\varepsilon_Nx_N).$

\begin{definition}\label{d:1} 
For a measurable function $f$ on $\bbR^N$, define
$$
    Uf(x)=\bigl(f(\varepsilon x)\bigr)_{\varepsilon\in\Sigma_N},\qquad x\in\calC.
$$
Thus $Uf$ is a $\bbC^{2^N}$-valued function on $\calC$.
\end{definition}

\begin{lemma} \label{l:3} 
For every $1\le p<\infty$,
\begin{equation}\label{e:8}
        \norm{f}_{L^p(\bbR^N,d\omega)}^p =\int_{\calC}\sum_{\varepsilon\in\Sigma_N}
        |f(\varepsilon x)|^p\,d\omega(x) =\norm{Uf}_{L^p(\calC,d\omega;\ell^p(\Sigma_N))}^p.
\end{equation}
For $p=\infty$,
$$
    \norm{f}_{L^\infty(\bbR^N)}=\norm{Uf}_{L^\infty(\calC;\ell^\infty(\Sigma_N))}.
$$
\end{lemma}

\begin{proof}
The sets $\varepsilon\calC$, $\varepsilon\in\Sigma_N$, are disjoint up to coordinate  hyperplanes, and those hyperplanes have $\omega$-measure zero. 
Since the weight $d\omega$ is invariant under sign changes,
$$
\int_{\varepsilon\calC}|f(x)|^p\,d\omega(x) =\int_{\calC}|f(\varepsilon x)|^p\,d\omega(x).
$$
Summing over $\varepsilon$ gives \eqref{e:8}. The $p=\infty$ statement follows in the same way.
\end{proof}

\begin{remark}
The lift does not assume that $f$ is even, radial, or invariant under the reflection group. 
If $f$ is arbitrary,  the vector $Uf(x)$ is arbitrary in $\bbC^{2^N}$. Radiality corresponds   only to a small subspace of this vector-valued space. 
The construction retains the whole vector.
\end{remark}

We next pass from sign values to Walsh parity coordinates.

The value lift records the values of $f$ on all chambers. It is often more  convenient to replace the sign-value basis by the parity basis.

For $\alpha=(\alpha_1,\ldots,\alpha_N)\in\{0,1\}^N$ and $\varepsilon\in\Sigma_N$, write
$$
\varepsilon^\alpha =\varepsilon_1^{\alpha_1}\cdots\varepsilon_N^{\alpha_N}.
$$

\begin{definition} \label{d:2} 
For $f$ on $\bbR^N$, define  its parity pieces on $\calC$ by
\begin{equation}\label{e:9}
        f_\alpha(x) =2^{-N}\sum_{\varepsilon\in\Sigma_N} \varepsilon^\alpha f(\varepsilon x), \qquad \alpha\in\{0,1\}^N.
\end{equation}
\end{definition}

The inverse formula is
\begin{equation}\label{e:10}
        f(\varepsilon x)=\sum_{\alpha\in\{0,1\}^N}\varepsilon^\alpha f_\alpha(x).
\end{equation}
This is   the finite Walsh transform on the group $\Sigma_N$.

\begin{lemma}
\label{l:4} For every $1\le p\le\infty$,
$$
\norm{Uf}_{L^p(\calC;\ell^p)}
        \simeq_{p,N}
        \norm{(f_\alpha)_\alpha}_{L^p(\calC;\ell^p)}.
$$
For $p=2$, the equality is   exact up to the constant $2^{N/2}$:
$$
\sum_{\varepsilon\in\Sigma_N}|f(\varepsilon x)|^2 =2^N\sum_{\alpha\in\{0,1\}^N}|f_\alpha(x)|^2.
$$
\end{lemma}

\begin{proof}
For each fixed $x$, the map
$$
     (f(\varepsilon x))_{\varepsilon\in\Sigma_N} \longmapsto (f_\alpha(x))_{\alpha\in\{0,1\}^N}
$$
is multiplication by a fixed invertible $2^N\times2^N$ matrix. All norms on the finite-dimensional space $\bbC^{2^N}$ are equivalent. 
This gives the first statement after integration in $x$. 
For $p=2$, the Walsh matrix is orthogonal after normalization by $2^{-N/2}$, giving the displayed equality.
\end{proof}

On the full space the same formula gives the usual parity projections.

The formula \eqref{e:9} is the chamber version of the usual parity projection. Define
$$
     P_\alpha f(x) =2^{-N}\sum_{\varepsilon\in\Sigma_N} \varepsilon^\alpha f(\varepsilon x).
$$
Then $P_\alpha f$ satisfies
$$
        P_\alpha f(\varepsilon x)=\varepsilon^\alpha P_\alpha f(x).
$$
Moreover
$$
        f=\sum_{\alpha\in\{0,1\}^N}P_\alpha f.
$$
The restriction of $P_\alpha f$ to $\calC$ is exactly $f_\alpha$.

\subsection{The product Dunkl transform in parity coordinates}
\label{s:5}

We recall the one-dimensional split.

We first recall the one-dimensional structure. Let $\kappa\ge0$. The one-dimensional Dunkl kernel has the form
\begin{equation}\label{e:11}
        E_\kappa(-ix,\xi) =j_{\kappa-1/2}(x\xi) - i\frac{x\xi}{2\kappa+1}j_{\kappa+1/2}(x\xi),
\end{equation}
where $j_\nu$ is the normalized Bessel function
$$
j_\nu(t)=2^\nu\Gamma(\nu+1)\frac{J_\nu(t)}{t^\nu}.
$$
The first term in \eqref{e:11} is even in both variables, and the second term is odd in both variables.

Let $f=f_0+f_1$ be the even-odd decomposition. On $(0,\infty)$, write
$$
    f_0(x)=F_0(x), \qquad f_1(x)=xF_1(x).
$$
Then the even part of the transform is a Hankel transform of order $\kappa-1/2$,  while the odd part is $-i\xi$ times a Hankel transform of order $\kappa+1/2$ applied to $F_1$. 
The exact constants depend only on $\kappa$. They are harmless in all $L^p$ estimates below.

This identity gives rise to the chamber matrix. A multiplier which is not even in $\xi$ mixes the even and odd components.

The tensor product of the one-dimensional split gives the product parity components.

For $\alpha\in\{0,1\}^N$, put
$$
        x^\alpha=x_1^{\alpha_1}\cdots x_N^{\alpha_N}.
$$
If $g$ belongs to the $\alpha$-parity component, then on $\calC$ we write
$$
    g(x)=x^\alpha G(x), \qquad G(x)=x^{-\alpha}g(x).
$$
The shifted Hankel measure is
\begin{equation}\label{e:12}
        d\omega_\alpha(x)  =x^{2\alpha}\,d\omega(x) =\prod_{j=1}^N x_j^{2\kappa_j+2\alpha_j}\,dx.
\end{equation}
Let $\calH_\alpha$ be the product Hankel transform with one-dimensional orders
$$
    \kappa_j+\alpha_j-1/2, \qquad j=1,\ldots,N.
$$
Then $\calH_\alpha$ is unitary on $L^2(\calC,d\omega_\alpha)$. Hence the conjugated transform
\begin{equation}\label{e:13}
        \calD_\alpha g(R) =R^\alpha\,\calH_\alpha(x^{-\alpha}g)(R)
\end{equation}
is unitary on $L^2(\calC,d\omega)$, since
\begin{equation}\label{e:14}
        \|g\|_{L^2(d\omega)}^2 =\int_\calC |x^{-\alpha}g(x)|^2\,d\omega_\alpha(x).
\end{equation}
All harmless normalizing constants from the one-dimensional Dunkl transform are absorbed into the definition of $\calD_\alpha$. 
With this convention the product Dunkl transform is a direct sum of the transforms $\calD_\alpha$ in Walsh parity coordinates.

\begin{proposition}\label{p:1}
Let $f\in L^2(\bbR^N,d\omega)$ and let $f_\alpha=(P_\alpha f)|_\calC$. 
If $(\calW U f)(x)=(f_\alpha(x))_\alpha,$ then the Walsh parity coordinates of the Dunkl transform are
\begin{equation}\label{e:15}
        (\calW U\Fkap f)_\alpha(R)=\calD_\alpha f_\alpha(R), \qquad R\in\calC.
\end{equation}
Consequently
\begin{equation}\label{e:16}
        \|f\|_{L^2(\bbR^N,d\omega)}^2 \simeq_{N,\kappa}\sum_{\alpha\in\{0,1\}^N} \|\calD_\alpha f_\alpha\|_{L^2(\calC,d\omega)}^2,
\end{equation}
and, with the normalization fixed above, the equivalence is equality up to the same finite Walsh factor as in Lemma \ref{l:4}.
\end{proposition}

\begin{proof}
In one dimension, \eqref{e:11} gives
$$
        E_\kappa(-ix,\xi) =E_{\kappa}^{(0)}(x\xi)+E_{\kappa}^{(1)}(x,\xi),
$$
where $E_{\kappa}^{(0)}$ is even in $(x,\xi)$ and $E_{\kappa}^{(1)}$ is odd in both variables. If $g_0$ is even and $g_1(x)=xG_1(x)$ is odd, then
$$
\Fkap g_0(\xi)\big|_{\xi>0} = c_{0,\kappa}\,\calH_0(g_0|_{(0,\infty)})(\xi)\quad\text{and}\quad
        \Fkap g_1(\xi)\big|_{\xi>0}= c_{1,\kappa}\,\xi\,\calH_1G_1(\xi),
$$
where $\calH_0$ and $\calH_1$ have orders $\kappa-1/2$ and $\kappa+1/2$, respectively. 
The factor $x\xi$ in the odd kernel changes $x^{2\kappa}\,dx$ into $x^{2\kappa+2}\,dx$ on the domain side and produces the factor $\xi$ on the range side.
Tensorizing this identity over $j=1,\ldots,N$ gives \eqref{e:15}. 
Plancherel theorem  for $\Fkap$, Lemma \ref{l:4}, and the unitarity of each $\calH_\alpha$ give \eqref{e:16}.
\end{proof}

We now record the finite Walsh algebra of a scalar symbol.

Let $m\in L^\infty(\bbR^N)$. For $R\in\calC$, define
\begin{equation}\label{e:17}
        m_\eta(R) =2^{-N}\sum_{\varepsilon\in\Sigma_N}\varepsilon^\eta m(\varepsilon R), \qquad \eta\in\{0,1\}^N.
\end{equation}
Then
\begin{equation}\label{e:18}
        m(\varepsilon R)=
        \sum_{\eta\in\{0,1\}^N}\varepsilon^\eta m_\eta(R).
\end{equation}

\begin{lemma}\label{l:5}
Let $h$ be a frequency-side function on $\bbR^N$, and suppose that in Walsh coordinates
$$
     h(\varepsilon R)=\sum_{\beta\in\{0,1\}^N}\varepsilon^\beta h_\beta(R), \qquad R\in\calC.
$$
Then the Walsh coordinates of $mh$ are
\begin{equation}\label{e:19}
        (mh)_\alpha(R)
        =\sum_{\beta\in\{0,1\}^N}
        m_{\alpha+\beta}(R)h_\beta(R),
\end{equation}
where $\alpha+\beta$ is taken modulo two in each coordinate. Equivalently, multiplication by $m$ is multiplication on $\calC$ by the finite matrix
\begin{equation}\label{e:20}
        \mathfrak m_{\alpha\beta}(R)=m_{\alpha+\beta}(R).
\end{equation}
\end{lemma}

\begin{proof}
Using \eqref{e:18},
$$
    m(\varepsilon R)h(\varepsilon R) =\sum_{\eta,\beta}\varepsilon^{\eta+\beta} m_\eta(R)h_\beta(R).
$$
Taking the $\alpha$-Walsh coefficient gives
$$
      (mh)_\alpha(R) =2^{-N}\sum_{\varepsilon\in\Sigma_N}  \varepsilon^\alpha m(\varepsilon R)h(\varepsilon R)
      =\sum_{\eta,\beta}m_\eta(R)h_\beta(R) 2^{-N}\sum_{\varepsilon\in\Sigma_N} \varepsilon^{\alpha+\eta+\beta}
      =\sum_{\beta}m_{\alpha+\beta}(R)h_\beta(R),
$$
because the finite Walsh average is $1$ when $\alpha+\eta+\beta=0$ and $0$ otherwise.
\end{proof}

\begin{remark}\label{r:1}
If $m$ is invariant under all sign changes, then $m_\eta=0$ for every $\eta\ne0$. Hence $\mathfrak m_{\alpha\beta}=0$ unless $\alpha=\beta$. 
Radial spectral multipliers form a subcase of this diagonal chamber situation. 
A scalar symbol with non-zero higher Walsh pieces produces off-diagonal entries; these entries are the components missed by a radial-function argument.
\end{remark}

Combining the parity decomposition and the Walsh matrix yields the lifted multiplier.

Let
$$
\mathcal U f=(\calD_\alpha f_\alpha)_\alpha, \qquad f_\alpha=(P_\alpha f)|_\calC.
$$
By Proposition \ref{p:1}, $\mathcal U$ is the Dunkl transform written in chamber Walsh coordinates. 
Therefore the multiplier $T_m=\Fkap^{-1}M_m\Fkap$ becomes
\begin{equation}\label{e:21}
        \mathbb T_m =\mathcal U^{-1}M_{\mathfrak m}\mathcal U,
        \qquad  (M_{\mathfrak m}H)_\alpha(R) =\sum_\beta m_{\alpha+\beta}(R)H_\beta(R).
\end{equation}
Equivalently, the entry from the $\beta$-source component to the $\alpha$-target component is
\begin{equation}\label{e:22}
        \mathbb T_{\alpha\beta} =\calD_\alpha^{-1}M_{m_{\alpha+\beta}}\calD_\beta.
\end{equation}
Thus
\begin{equation}\label{e:23}
        \|\mathbb T_{\alpha\beta}\|_{L^2(\calC,d\omega)\to L^2(\calC,d\omega)} \le \|m_{\alpha+\beta}\|_{L^\infty(\calC)}.
\end{equation}
The full $L^2$ norm is better read before taking entries. 
For each fixed $R$, the Walsh matrix $\mathfrak m(R)$ is unitarily conjugate to the diagonal matrix
$$
        \operatorname{diag}\{m(\varepsilon R):\varepsilon\in\Sigma_N\}.
$$
Therefore
\begin{equation}\label{e:24}
        \|M_{\mathfrak m}\|_{L^2(\calC;\ell^2)\to L^2(\calC;\ell^2)} =\operatorname*{ess\,sup}_{R\in\calC}
          \max_{\varepsilon\in\Sigma_N}|m(\varepsilon R)| =\|m\|_\infty.
\end{equation}
By \eqref{e:21}, this is exactly the ordinary Dunkl Plancherel estimate for $T_m$.

For dyadic estimates, fix $\psi_j(r)=\psi(2^{-j}r)$ and set
\begin{equation}\label{e:25}
        \mathbb T_{\alpha\beta,j}  =\calD_\alpha^{-1}M_{\psi_j m_{\alpha+\beta}}\calD_\beta.
\end{equation}
Its off-diagonal kernel is denoted by $K_j^{\alpha\beta}(x,y)$:
$$
     \mathbb T_{\alpha\beta,j}F(x) =\int_\calC K_j^{\alpha\beta}(x,y)F(y)\,d\omega(y), \qquad x\notin\operatorname{supp}F.
$$
The scalar verification in Section \ref{s:9} consists in estimating the kernels generated by \eqref{e:25}, 
with the symbol $\psi_jm_{\alpha+\beta}$.


\section{Chamber H\"ormander estimates and matrix Calder\'on--Zygmund theory}
\label{s:6}

\subsection{Dyadic decomposition}

Fix $\psi\in C_c^\infty((1/2,2))$, $\psi\ge0$, with
$$
   \sum_{j\in\bbZ}\psi(2^{-j}r)=1, \qquad r>0.
$$
Set $\psi_j(r)=\psi(2^{-j}r)$ and $r_j=2^{-j}$. The homogeneous dyadic pieces of the lifted multiplier are
$$
\mathbb T_{m,j}:\quad \text{symbol } \psi_j(|\xi|)m(\xi),
        \qquad \mathbb T_m=\text{$L^2$-}\lim_{J\to\infty}\sum_{|j|\le J}\mathbb T_{m,j}.
$$
There is no free low-frequency term in this convention. 
If one replaces this by an inhomogeneous partition, the multiplier supported near the origin must satisfy the same kernel estimates, or an additional smoothness condition. 
A compactly supported $L^\infty$ symbol near $0$ is not by itself an $L^p$ multiplier.

For an entry of the chamber matrix, write the off-diagonal kernel as
$$
\mathbb T_{\alpha\beta,j}F(x) =\int_\calC K_j^{\alpha\beta}(x,y)F(y)\,d\omega(y),  \qquad x\notin\operatorname{supp}F.
$$
All entries are measured with respect to the same chamber measure $d\omega$.

\subsection{Definition of the \texorpdfstring{$\CH^2_{s,\eta}$}{CH2seta} condition}

The next condition is the dyadic condition used in the paper.  It consists of two weighted $L^2$ size estimates and two first-difference estimates. 
The threshold $s>N_\kappa/2$ enters only when these estimates are converted into the ordinary integral H\"ormander condition.

\begin{definition}
\label{d:3} 
Let $s>0$ and $0<\eta\le1$. We say that $\mathbb T_m\in\CH^2_{s,\eta}$ if there is a constant $A$ such that, for all $j\in\bbZ$, 
all $\alpha,\beta\in\{0,1\}^N$, and all admissible points $x,y,x',y'\in\calC$,
\begin{equation}\label{e:26}
        V(y,r_j)^{1/2}
        \left( \int_\calC \left(1+\frac{\rho(x,y)}{r_j}\right)^{2s}  |K_j^{\alpha\beta}(x,y)|^2\,d\omega(x) \right)^{1/2}\le A,
\end{equation}
\begin{equation}\label{e:27}
        V(x,r_j)^{1/2}
        \left( \int_\calC  \left(1+\frac{\rho(x,y)}{r_j}\right)^{2s} |K_j^{\alpha\beta}(x,y)|^2\,d\omega(y)   \right)^{1/2}\le A,
\end{equation}
\begin{align}
        &V(y,r_j)^{1/2} \left( \int_\calC \left(1+\frac{\rho(x,y)}{r_j}\right)^{2s}  |K_j^{\alpha\beta}(x,y)-K_j^{\alpha\beta}(x,y')|^2
        \,d\omega(x) \right)^{1/2}       \notag\\
        & \le A\left(\frac{\rho(y,y')}{r_j}\right)^\eta,  \qquad \rho(y,y')\le r_j \label{e:28},
\end{align}
\begin{align}
        &V(x,r_j)^{1/2} \left(\int_\calC \left(1+\frac{\rho(x,y)}{r_j}\right)^{2s} |K_j^{\alpha\beta}(x,y)-K_j^{\alpha\beta}(x',y)|^2 \,d\omega(y)\right)^{1/2}       \notag\\
        & \le A\left(\frac{\rho(x,x')}{r_j}\right)^\eta, \qquad \rho(x,x')\le r_j \label{e:29}.
\end{align}
Finally, the dyadic kernels must represent the off-diagonal kernel. Equivalently, for bounded compactly supported $F$ and $x\notin\operatorname{supp}F$,
$$
     \sum_{|j|\le J}\int_\calC K_j^{\alpha\beta}(x,y)F(y)\,d\omega(y) \longrightarrow \int_\calC K^{\alpha\beta}(x,y)F(y)\,d\omega(y),
$$
where $K^{\alpha\beta}(x,y)=\sum_{j\in\bbZ}K_j^{\alpha\beta}(x,y)$ for $x\ne y$. 
Equivalently, one may prove the estimates for finite dyadic truncations with constants independent of the truncation and then pass to the $L^2$ limit. 
The least admissible $A$, after the maximum over all matrix entries is taken, is written $\|\mathbb T_m\|_{\CH^2_{s,\eta}}$.
\end{definition}

\begin{remark}
The condition is imposed on kernels, not merely on the scalar symbol. 
After the chamber lift, a non-radial multiplier becomes a finite matrix of Hankel operators between parity components. 
The Calder\'on--Zygmund argument uses these matrix kernels.
\end{remark}

\subsection{\texorpdfstring{Dyadic $L^2$ estimates imply $L^1$ estimates}{Dyadic L2 estimates imply L1 estimates}}

\begin{lemma}
\label{l:6} 
Let $s>N_\kappa/2$. If \eqref{e:26} holds for $K_j$, then, for every $y\in\calC$,
\begin{equation}\label{e:30}
        \int_\calC |K_j(x,y)|\,d\omega(x)\le CA,
\end{equation}
and, for every $a\ge1$,
\begin{equation}\label{e:31}
        \int_{\rho(x,y)>a r_j}|K_j(x,y)|\,d\omega(x)  \le CAa^{N_\kappa/2-s}.
\end{equation}
The estimates in the other variable follow from \eqref{e:27}.
\end{lemma}

\begin{proof}
Set $W_j(x,y)=(1+\rho(x,y)/r_j)^s$. For every measurable $E\subset\calC$, Cauchy--Schwarz gives
\begin{equation}\label{e:32}
        \int_E |K_j(x,y)|\,d\omega(x) \le  \left(\int_E W_j(x,y)^2|K_j(x,y)|^2\,d\omega(x)\right)^{1/2} \left(\int_E W_j(x,y)^{-2}\,d\omega(x)\right)^{1/2}.
\end{equation}
For $E=\calC$, the first factor is $\le AV(y,r_j)^{-1/2}$ and the second is $\le CV(y,r_j)^{1/2}$ by Lemma \ref{l:2}. 
This proves \eqref{e:30}. If $E=\{x:\rho(x,y)>ar_j\}$, then
$$
\int_E W_j(x,y)^{-2}\,d\omega(x) \le CV(y,r_j)a^{N_\kappa-2s},
$$
again by Lemma \ref{l:2}; inserting this into \eqref{e:32} gives \eqref{e:31}.
\end{proof}

\begin{lemma}
\label{l:7} Let $s>N_\kappa/2$ and $0<\eta\le1$. If \eqref{e:28} holds and $\rho(y,y')\le
r_j$, then
\begin{equation}\label{e:33}
        \int_\calC |K_j(x,y)-K_j(x,y')|\,d\omega(x)
        \le CA\left(\frac{\rho(y,y')}{r_j}\right)^\eta.
\end{equation}
The symmetric estimate follows from \eqref{e:29}.
\end{lemma}

\begin{proof}
Use \eqref{e:32} with $E=\calC$ and replace $K_j(x,y)$ by $K_j(x,y)-K_j(x,y')$. The first factor is controlled by
\eqref{e:28}; the second is controlled by Lemma \ref{l:2}.
\end{proof}

\subsection{The global H\"ormander integral condition}

Let $K(x,y)=\sum_{j\in\bbZ}K_j(x,y)$ off the diagonal. The following scale summation is the passage from $\CH^2_{s,\eta}$ to
the usual integral H\"ormander kernel condition.

\begin{proposition}
\label{p:2} Assume $s>N_\kappa/2$ and $0<\eta\le1$. Suppose the dyadic kernels satisfy
\eqref{e:26} and \eqref{e:28}. Then
\begin{equation}\label{e:34}
        \sup_{y,y'\in\calC}
        \int_{\rho(x,y)>2\rho(y,y')}
        |K(x,y)-K(x,y')|\,d\omega(x)
        \le CA.
\end{equation}
If \eqref{e:27} and \eqref{e:29} also hold, then
\begin{equation}\label{e:35}
        \sup_{x,x'\in\calC}
        \int_{\rho(x,y)>2\rho(x,x')}
        |K(x,y)-K(x',y)|\,d\omega(y)
        \le CA.
\end{equation}
\end{proposition}

\begin{proof}
We prove \eqref{e:34}. Let $\delta=\rho(y,y')$. If $\delta=0$, there is nothing to show. Choose $j_0\in\bbZ$ such
that
$$
2^{-j_0-1}<\delta\le2^{-j_0}.
$$
For $j\le j_0$ one has $\delta\le r_j$, and Lemma \ref{l:7} gives
$$
\sum_{j\le j_0}
        \int_{\rho(x,y)>2\delta}|K_j(x,y)-K_j(x,y')|\,d\omega(x)
        \le CA\sum_{j\le j_0}\left(\frac{\delta}{r_j}\right)^\eta
        \le CA\sum_{k=0}^\infty2^{-k\eta}
        \le C_\eta A.
$$
For $j>j_0$, the region $\rho(x,y)>2\delta$ implies $\rho(x,y')>\delta$. Hence Lemma \ref{l:6} gives
\begin{align*}
        \int_{\rho(x,y)>2\delta}|K_j(x,y)-K_j(x,y')|\,d\omega(x)
        &\le
        \int_{\rho(x,y)>2\delta}|K_j(x,y)|\,d\omega(x)
        +\int_{\rho(x,y')>\delta}|K_j(x,y')|\,d\omega(x)          \\
        &\le CA\left(\frac{r_j}{\delta}\right)^{s-N_\kappa/2}.
\end{align*}
Since $s>N_\kappa/2$,
$$
\sum_{j>j_0}\left(\frac{r_j}{\delta}\right)^{s-N_\kappa/2}
        \le C\sum_{k=1}^\infty2^{-k(s-N_\kappa/2)}<\infty.
$$
The two estimates prove \eqref{e:34}; \eqref{e:35} is identical after interchanging the variables.
\end{proof}


\subsection{Finite matrix Calder\'on--Zygmund theorem on the chamber}
\label{s:7}

We use the following standard singular integral hypotheses on the chamber.

Let $(\calC,\rho,d\omega)$ be the chamber space and let $E$ be a finite-dimensional complex vector space. A chamber singular
integral is a linear operator $\mathbb T$ with kernel $K(x,y)\in\calL(E)$, $x\ne y$, such that
$$
    \mathbb T F(x)=\int_\calC K(x,y)F(y)\,d\omega(y), \qquad x\notin\operatorname{supp}F,
$$
for bounded compactly supported $E$-valued functions $F$. We assume
\begin{equation}\label{e:36}
        \sup_{y,y'} \int_{\rho(x,y)>2\rho(y,y')}\|K(x,y)-K(x,y')\|_{\calL(E)}\,d\omega(x)<\infty,
\end{equation}
\begin{equation}\label{e:37}
        \sup_{x,x'}
        \int_{\rho(x,y)>2\rho(x,x')}
        \|K(x,y)-K(x',y)\|_{\calL(E)}\,d\omega(y)<\infty.
\end{equation}

\begin{definition}
\label{d:4} 
We say that $\mathbb T$ has admissible truncations if
$$
\mathbb T_\varepsilon F(x)=
        \int_{\rho(x,y)>\varepsilon}K(x,y)F(y)\,d\omega(y)
$$
are uniformly bounded on $L^2(\calC,d\omega;E)$, or if the principal value operator is the $L^2$-limit of such uniformly bounded truncations.
\end{definition}

The finite-dimensional matrix estimate follows entry by entry from the scalar Calder\'on--Zygmund theorem.

\begin{theorem}
\label{t:6} Let $E$ be finite-dimensional. Suppose $\mathbb T$ is bounded on $L^2(\calC,d\omega;E)$ and its kernel
satisfies \eqref{e:36} and \eqref{e:37}. Then
$$
    \|\mathbb T F\|_{L^p(\calC;E)}  \le C_p\|F\|_{L^p(\calC;E)},  \qquad 1<p<\infty.
$$
If $\mathbb T$ has admissible truncations, then
$$
\omega\{x:\|\mathbb T F(x)\|_E>\lambda\}  \le \frac{C}{\lambda}\|F\|_{L^1(\calC;E)}.
$$
\end{theorem}

\begin{proof}
Choose a basis $e_1,\ldots,e_M$ of $E$ and write
$$
     F=\sum_{b=1}^M F_b e_b,  \qquad  \mathbb T F=\sum_{a=1}^M\left(\sum_{b=1}^M T_{ab}F_b\right)e_a.
$$
The scalar entries satisfy
$$
        \|T_{ab}\|_{L^2\to L^2} \le C_E\|\mathbb T\|_{L^2(E)\to L^2(E)}.
$$
Moreover \eqref{e:36}--\eqref{e:37} imply the scalar H\"ormander conditions for each $T_{ab}$. 
Hence, by the standard scalar Calder\'on--Zygmund theorem on the space of homogeneous type $(\calC,\rho,d\omega)$ \cite{CW1971},
$$
\|T_{ab}g\|_{L^p(\calC)}\le C_p\|g\|_{L^p(\calC)},
        \qquad 1<p<\infty.
$$
Since the matrix is finite,
\begin{align*}
        \|\mathbb T F\|_{L^p(\calC;E)}  &\le C_E\sum_{a,b=1}^M\|T_{ab}F_b\|_{L^p(\calC)}       
         \le C_{p,E}\sum_{b=1}^M\|F_b\|_{L^p(\calC)} \le C_{p,E}\|F\|_{L^p(\calC;E)}.
\end{align*}
The weak type estimate is obtained from the scalar weak type $(1,1)$ theorem entry by entry and summing over the finite matrix.
\end{proof}

We next apply this matrix theorem to kernels satisfying the chamber $L^2$ H\"ormander estimates.

\begin{proposition}
\label{p:3} 
Let $s>N_\kappa/2$ and $0<\eta\le1$. Let $\mathbb T$ be an $L^2$-bounded finite matrix operator on the chamber. 
If the dyadic kernels of every entry satisfy Definition \ref{d:3}, then $\mathbb T$ is bounded on $L^p(\calC,d\omega;E)$ for every $1<p<\infty$. 
If the truncations are admissible, then $\mathbb T$ is of weak type $(1,1)$.
\end{proposition}

\begin{proof}
Proposition \ref{p:2} gives, for every scalar entry,
$$
     H_y(K_{ab})+H_x(K_{ab}) \le C_{s,\eta,N,\kappa}  \|\mathbb T\|_{\CH^2_{s,\eta}}.
$$
Since $E$ is finite-dimensional, the operator norm of a matrix is bounded by a fixed multiple of the largest scalar entry.
The same comparison holds in the first variable. Hence \eqref{e:36}--\eqref{e:37} hold, and Theorem \ref{t:6} applies.
\end{proof}

\subsection{Proof of the product chamber theorem}
\label{s:8}

\begin{proof}[Proof of Theorem \ref{t:2}]
Let $f\in L^2(d\omega)\cap L^p(d\omega)$ and set $F=\calW Uf$. 
By Lemmas \ref{l:3} and \ref{l:4}, $\|F\|_{L^p(\calC;\bbC^{2^N})}\simeq_N \|f\|_{L^p(\bbR^N,d\omega)}$, and $\calW U(T_mf)=\mathbb T_mF$. 
Dunkl Plancherel gives
$$
    \|\mathbb T_m\|_{L^2(\calC;\bbC^{2^N})\to L^2(\calC;\bbC^{2^N})}  \le C_N\|m\|_\infty.
$$
The condition $\mathbb T_m\in\CH^2_{s,\eta}$ and Proposition \ref{p:3} imply
$$
     \|\mathbb T_mF\|_{L^p(\calC;\bbC^{2^N})}  \le C_{p,N,\kappa,s,\eta} \bigl(\|m\|_\infty+ \|\mathbb T_m\|_{\CH^2_{s,\eta}}\bigr)   \|F\|_{L^p(\calC;\bbC^{2^N})}.
$$
Transferring back gives
$$
     \|T_m f\|_{L^p(\bbR^N,d\omega)}  \le  C_{p,N,\kappa,s,\eta}  \bigl(\|m\|_\infty+  \|\mathbb T_m\|_{\CH^2_{s,\eta}}\bigr) \|f\|_{L^p(\bbR^N,d\omega)}.
$$
The dense subspace $L^2\cap L^p$ gives the extension to all of $L^p$, $1<p<\infty$. 
The proof takes place on the full chamber vector-valued space, not on a radial or invariant subspace.
\end{proof}

\begin{remark}
A radial-function multiplier theorem does not imply a full non-radial theorem by duality, since a non-radial multiplier need not preserve the radial subspace. 
The chamber lift replaces the full-space question with a finite-dimensional vector-valued singular integral estimate on one chamber.
\end{remark}


\section{Scalar Walsh--Sobolev verification}
\label{s:9}

The remaining step is to verify $\CH^2_{s,\eta}$ from scalar conditions on $m$.
 In the product case this is possible because the Dunkl transform splits into tensor-product parity components. 
 The $(\alpha,\beta)$ matrix entry is a product Bessel oscillatory integral with amplitude $m_{\alpha+\beta}$. 
 The verification is deliberately made in $L^2$; no pointwise kernel bound is used. 
 This is the argument which keeps the H\"ormander order at the scale $N_\kappa/2$.

The scalar theorem remains wall-separated in frequency. The dyadic Walsh pieces are kept a fixed distance from the coordinate faces $R_i=0$. 
There is no restriction on the spatial variables and no radiality condition on $f$.

Throughout this section $\Sigma_N\simeq\bbZ_2^N$, $\calC=(0,\infty)^N$, and
$$
     d\omega(x)=\prod_{i=1}^N x_i^{2\kappa_i}\,dx,  \qquad N_\kappa=N+2\sum_{i=1}^N\kappa_i.
$$
Write $V(x,r)=\omega(B(x,r))$, where $B(x,r)=\{y\in\calC:|x-y|<r\}$. The volume estimate used below is
\begin{equation}\label{e:38}
        V(x,r)\simeq r^N\prod_{i=1}^N(x_i+r)^{2\kappa_i}.
\end{equation}

\subsection{Kernels between parity components}

For one coordinate and one multiplicity parameter $\lambda\ge0$, set
$$
      \varphi^{(\lambda)}_0(t) = j_{\lambda-1/2}(t), \qquad \varphi^{(\lambda)}_1(t)
      =  \frac{t}{2\lambda+1}j_{\lambda+1/2}(t),  \qquad t\ge0,
$$
where $j_\nu$ is the normalized Bessel function
$$
      j_\nu(t)=2^\nu\Gamma(\nu+1)\frac{J_\nu(t)}{t^\nu}.
$$
For $\alpha\in\{0,1\}^N$, define
\begin{equation}\label{e:39}
        \Phi_\alpha(x,R) = \prod_{i=1}^N  \varphi^{(\kappa_i)}_{\alpha_i}(x_iR_i),  \qquad x,R\in\calC.
\end{equation}

As explained in Section \ref{s:5}, the odd component is equivalently a shift of the Hankel measure. Namely,
$$
     d\omega_\alpha(x) =\prod_{i=1}^N x_i^{2\kappa_i+2\alpha_i}\,dx,  \qquad
F_\alpha(x)=x^{-\alpha}f_\alpha(x),
$$
and
\begin{equation}\label{e:40}
        \|f_\alpha\|_{L^2(d\omega)} = \|F_\alpha\|_{L^2(d\omega_\alpha)}.
\end{equation}
Thus all kernels below are written with respect to the same base measure $d\omega$, but the factors in $\Phi_\alpha$ retain the information of this shifted Hankel bookkeeping.

For a symbol $m\in L^\infty(\bbR^N)$, recall its Walsh pieces on the positive frequency chamber:
\begin{equation}\label{e:41}
        m_\theta(R) = 2^{-N} \sum_{\varepsilon\in\{\pm1\}^N}\varepsilon^\theta m(\varepsilon R),\qquad  \theta\in\{0,1\}^N.
\end{equation}
Then
$$
        m(\varepsilon R) = \sum_{\theta\in\{0,1\}^N} \varepsilon^\theta m_\theta(R).
$$
In the $(\alpha,\beta)$-component, the multiplier piece is $m_{\alpha+\beta}$, where addition is modulo two. Let
$$
    \psi_j(r)=\psi(2^{-j}r), \qquad r_j=2^{-j},
$$
and put
\begin{equation}\label{e:42}
        a_{\theta,j}(R)=\psi_j(|R|)m_\theta(R).
\end{equation}
The dyadic chamber kernel of the $(\alpha,\beta)$-entry is
\begin{equation}\label{e:43}
        K_j^{\alpha\beta}(x,y) =  c_{\alpha\beta,\kappa}  \int_\calC  a_{\alpha+\beta,j}(R)  \Phi_\alpha(x,R)  \Phi_\beta(y,R) \,d\omega(R),
\end{equation}
where $c_{\alpha\beta,\kappa}$ is harmless and non-zero.

\subsection{Wall separation and scalar Sobolev norms}

Fix a parameter $0<c_0<N^{-1/2}$ and define the wall-separated subchamber of frequency space by
\begin{equation}\label{e:44}
        \Gamma_{c_0}  =  \{R\in\calC: R_i\ge c_0 |R|\text{ for every }i=1,\ldots,N\}.
\end{equation}
Since $m_\theta$ is initially only an $L^\infty$-function, support below means essential support. The support condition used in this section is
\begin{equation}\label{e:45}
        \operatorname*{ess\,supp}m_\theta\subset \Gamma_{c_0}   \qquad\text{for every }\theta\in\{0,1\}^N.
\end{equation}
It excludes the delicate case where a dyadic    annulus touches a coordinate wall $R_i=0$.  No spatial wall separation is  imposed.

For $\sigma\ge0$, let $W_2^\sigma(\bbR^N)$ denote the usual Bessel-potential Sobolev space. Each dyadic piece  $a_{\theta,j}(2^j\cdot)$ is extended by zero outside $\calC$. Define
\begin{equation}\label{e:46}
        \calS_\sigma(m) =   \max_{\theta\in\{0,1\}^N}   \sup_{j\in\bbZ}  \left\|a_{\theta,j}(2^j\cdot)\right\|_{W_2^\sigma(\bbR^N)}.
\end{equation}
This is the scalar dyadic Sobolev quantity which will be lifted to the chamber matrix estimate.

For comparison with the usual Mihlin formulation, for an integer $M\ge0$ set
\begin{equation}\label{e:47}
        \calM_M(m)  :=\max_{\theta\in\{0,1\}^N} \max_{|\nu|\le M} \sup_{R\in\calC}  |R|^{|\nu|}\,|\partial_R^\nu m_\theta(R)|.
\end{equation}
Finiteness means that each $m_\theta$ is $C^M$ on $\calC$, with the weighted derivative bounds above.

\begin{lemma}
\label{l:8} Assume \eqref{e:45} and $\calM_M(m)<\infty$ for an integer $M\ge0$.  If $0\le\sigma<M$, then
\begin{equation}\label{e:48}
        \calS_\sigma(m) \le C_{N,M,\sigma,c_0,\psi}\,\calM_M(m).
\end{equation}
\end{lemma}

\begin{proof}
For $b_{\theta,j}(u)=a_{\theta,j}(2^ju)$, the support is contained in the fixed compact set
$$
     \Gamma_{c_0}\cap\{1/2\le |u|\le2\}.
$$
For every $|\nu|\le M$, it follows from Leibniz' rule to see
$$
        |\partial_u^\nu b_{\theta,j}(u)|   \le C_{\nu,\psi}\calM_M(m)
$$
uniformly in $j$ and $\theta$. Hence the integer Sobolev norms $\|b_{\theta,j}\|_{W_2^k}$, $0\le k\le M$, are uniformly bounded by $C\calM_M(m)$. 
The fractional estimate for every $\sigma<M$ follows from the standard interpolation description of Bessel-potential Sobolev spaces on $\bbR^N$.
\end{proof}

\subsection{A Fourier-Sobolev lemma for compact frequency support}

The product Bessel proof uses only one elementary Fourier fact. 
Compact frequency support allows one to freeze some physical variables without paying an extra Sobolev order. This is the point that removes an unnecessary $N/2$ loss.

\begin{lemma}
\label{l:9} 
Let $K\Subset\bbR^N$, let $0\le s<\sigma$, and let $N_0>N/2+2$. Suppose that $q(w,u)$ is a smooth amplitude on $\bbR^N_w\times\bbR^N_u$, supported in $u\in K$, and satisfying
\begin{equation}\label{e:49}
        \sup_{w\in\bbR^N,\,u\in K}  |\partial_w^\mu\partial_u^\nu q(w,u)| \le C_{\mu\nu}
\end{equation}
for all $|\mu|,|\nu|\le N_0+\lceil\sigma\rceil+2$. For $b\in W_2^\sigma(\bbR^N)$, set
$$
T_qb(w)=\int_{\bbR^N}e^{iw\cdot u}q(w,u)b(u)\,du.
$$
Then
\begin{equation}\label{e:50}
        \left\|(1+|w|)^sT_qb(w)\right\|_{L^2_w(\bbR^N)}  \le C\|b\|_{W_2^\sigma(\bbR^N)}.
\end{equation}
Moreover, for every subset $J\subset\{1,\ldots,N\}$,
\begin{equation}\label{e:51}
        \sup_{w_{J^c}\in\bbR^{J^c}} \left( \int_{\bbR^J}(1+|w|)^{2s}|T_qb(w)|^2\,dw_J  \right)^{1/2}  \le C\|b\|_{W_2^\sigma(\bbR^N)}.
\end{equation}
The constant depends on $N,s,\sigma,K$ and finitely many constants in \eqref{e:49}, but not on $b$.
\end{lemma}

\begin{proof}
We first prove the unweighted $L^2$ estimate. The kernel of $T_qT_q^*$ is
\begin{align*}
        \calK(w,w')  &= \int_K e^{i(w-w')\cdot u}  q(w,u)\overline{q(w',u)}\,du.
\end{align*}
Choose an integer $L>N/2$. Since
$$
     (1-\Delta_u)^L e^{i(w-w')\cdot u} =(1+|w-w'|^2)^L e^{i(w-w')\cdot u}.
$$
Integration by parts gives
\begin{align*}
        |\calK(w,w')|  &\le  C_L(1+|w-w'|^2)^{-L}   \sum_{|\gamma|\le 2L}
        \int_K  |\partial_u^\gamma(q(w,u)\overline{q(w',u)})|\,du    \le C_L(1+|w-w'|)^{-2L}.
\end{align*}
The last kernel is integrable in both variables. Schur's lemma yields
\begin{equation}\label{e:52}
        \|T_qb\|_{L^2_w}\le C\|b\|_{L^2_u}.
\end{equation}

Let $n$ be a non-negative integer. For $|\mu|\le n$,
$$
        w^\mu e^{iw\cdot u}=(-i)^{|\mu|}\partial_u^\mu e^{iw\cdot u}.
$$
After integrating by parts in $u$,
\begin{align*}
        w^\mu T_qb(w) &=  \sum_{\nu\le\mu} c_{\mu\nu}  \int_K e^{iw\cdot u}  \partial_u^{\mu-\nu}q(w,u)\,\partial_u^\nu b(u)\,du.
\end{align*}
Each amplitude $\partial_u^{\mu-\nu}q$ satisfies the same finite support and symbol bounds. Applying \eqref{e:52} to every term gives
$$
     \sum_{|\mu|\le n}\|w^\mu T_qb\|_2 \le C_n\sum_{|\nu|\le n}\|\partial_u^\nu b\|_2  \le C_n\|b\|_{W^n_2}.
$$
Since $(1+|w|)^n\lesssim 1+\sum_{1\le|\mu|\le n}|w^\mu|$, we obtain
$$
     \|(1+|w|)^nT_qb\|_2\le C_n\|b\|_{W^n_2}.
$$
Complex interpolation between adjacent integer orders proves \eqref{e:50} for every $0\le  s<\sigma$.

For \eqref{e:51}, fix $J\subset\{1,\ldots,N\}$ and put $J^c=\{1,\ldots,N\}\setminus J$. 
Choose an integer $m>|J^c|/2$. 
The Sobolev embedding $W^m_2(\bbR^{J^c})\hookrightarrow L^\infty(\bbR^{J^c})$, applied to the function $w_{J^c}\mapsto (1+|w|)^sT_qb(w)$ and then integrated in $w_J$, gives
\begin{align*}
        \sup_{w_{J^c}}  \int_{\bbR^J}(1+|w|)^{2s}|T_qb(w)|^2\,dw_J
        &\le   C\sum_{|\rho|\le m}  \int_{\bbR^N}  \left| \partial_{w_{J^c}}^\rho\bigl[(1+|w|)^sT_qb(w)\bigr]   \right|^2\,dw.
\end{align*}
A derivative in $w_{J^c}$ has only two effects. If it falls on the weight, the weight order drops from $s$ to at most $s$. If
it falls on $T_qb$, then
$$
    \partial_{w_\ell}T_qb(w) = \int e^{iw\cdot u} \bigl(iu_\ell q(w,u)+\partial_{w_\ell}q(w,u)\bigr)b(u)\,du,
$$
and higher derivatives have the same form with another admissible compactly supported amplitude, because $u\in K$. 
Hence every term on the right is bounded by \eqref{e:50}, with $q$ replaced by one of finitely many amplitudes satisfying \eqref{e:49}. 
This proves the mixed estimate.
\end{proof}

\subsection{Product Bessel Sobolev localization}

We now prove the product Bessel estimate needed for the chamber condition. The result is an $L^2$-estimate, not a pointwise
estimate.

\begin{lemma}
\label{l:10} Let $\lambda\ge0$, let $a\in\{0,1\}$, and let $L\ge0$. 
There are $\chi\in  C_c^\infty([0,2))$, with $\chi=1$ on $[0,1]$, and smooth functions $e_a^{(\lambda)}$, $b_{a,+}^{(\lambda)}$, $b_{a,-}^{(\lambda)}$ such that, for $t\ge0$,
\begin{equation}\label{e:53}
        \varphi_a^{(\lambda)}(t) =  \chi(t)e_a^{(\lambda)}(t)  +(1-\chi(t)) \sum_{\sigma=\pm1}e^{i\sigma t}b_{a,\sigma}^{(\lambda)}(t),
\end{equation}
and
\begin{equation}\label{e:54}
        \left|\frac{d^\ell}{dt^\ell}e_a^{(\lambda)}(t)\right|  \le C_{\ell,\lambda}, \qquad 0\le t\le2,
\end{equation}
while
\begin{equation}\label{e:55}
        \left|\frac{d^\ell}{dt^\ell}b_{a,\sigma}^{(\lambda)}(t)\right|  \le C_{\ell,\lambda}(1+t)^{-\lambda-\ell},  \qquad t\ge1,
\end{equation}
for every $0\le\ell\le L$ and $\sigma=\pm1$.
\end{lemma}

\begin{proof}
Near zero, $j_{\lambda-1/2}$ and $t j_{\lambda+1/2}(t)$ are smooth functions of $t$, so the local part is immediate. 
For $t\ge1$, the classical asymptotic expansion of $J_\nu$, differentiated term by term a finite number of times, gives
$$
     J_\nu(t)=e^{it}c_{\nu,+}(t)+e^{-it}c_{\nu,-}(t),  \qquad |c_{\nu,\sigma}^{(\ell)}(t)|\le C_{\ell,\nu}t^{-1/2-\ell}.
$$
Since $j_{\lambda-1/2}(t)=C_\lambda t^{-\lambda+1/2}J_{\lambda-1/2}(t)$ and 
$t j_{\lambda+1/2}(t)=C_\lambda t^{-\lambda+1/2}J_{\lambda+1/2}(t)$, both large-argument amplitudes obey \eqref{e:55}.
Multiplying by $1-\chi$ preserves these estimates.
\end{proof}

\begin{proposition}
\label{p:4} 
Let $0<c_0<N^{-1/2}$, let $0<s<\sigma$, and let $a_j$ be measurable functions on $\calC$ such that
\begin{equation}\label{e:56}
        \operatorname*{ess\,supp}a_j  \subset \Gamma_{c_0}\cap\{R\in\calC:2^{j-1}\le |R|\le2^{j+1}\}.
\end{equation}
Assume that, after extension by zero outside $\calC$,
\begin{equation}\label{e:57}
        \sup_{j\in\bbZ} \left\|a_j(2^j\cdot)\right\|_{W_2^\sigma(\bbR^N)} \le A.
\end{equation}
Fix $\alpha,\beta\in\{0,1\}^N$ and define
\begin{equation}\label{e:58}
        L_j^{\alpha\beta}(x,y) = \int_\calC a_j(R)\Phi_\alpha(x,R)\Phi_\beta(y,R) \,d\omega(R).
\end{equation}
Then, with $r_j=2^{-j}$,
\begin{equation}\label{e:59}
        V(y,r_j)^{1/2} \left( \int_\calC  \left(1+\frac{|x-y|}{r_j}\right)^{2s}  |L_j^{\alpha\beta}(x,y)|^2\,d\omega(x) \right)^{1/2}  \le C A.
\end{equation}
The symmetric estimate in the $x$-variable also holds. Moreover, if $|y-y'|\le r_j$, then
\begin{align}
        &V(y,r_j)^{1/2} \left(  \int_\calC  \left(1+\frac{|x-y|}{r_j}\right)^{2s}  |L_j^{\alpha\beta}(x,y)-L_j^{\alpha\beta}(x,y')|^2  \,d\omega(x)  \right)^{1/2} 
        \le  C A\left(\frac{|y-y'|}{r_j}\right) \label{e:60}.
\end{align}
The symmetric first difference in the $x$-variable also holds. The constant is independent of $j,A,x,y$.
\end{proposition}

\begin{proof}
We prove the $y$-size and $y$-difference estimates. The $x$ estimates follow from the same argument with the two Bessel factors interchanged. 
The proof is written with the dyadic scaling visible, because this is the point where the frequency scale $2^j$, the physical scale $r_j=2^{-j}$, and the chamber volume normalization meet.

\noindent\textit{Step 1: exact dyadic rescaling.} Put
$$
   R=2^ju, \qquad  \widetilde x=2^jx, \qquad \widetilde y=2^jy, \qquad \widetilde y'=2^jy', \qquad  b_j(u)=a_j(2^ju).
$$
Then $\operatorname*{ess\,supp} b_j\subset K_{c_0}:=\Gamma_{c_0}\cap\{u\in\calC:1/2\le |u|\le2\}$ and $\|b_j\|_{W_2^\sigma}\le A$. 
Here we use that $\Gamma_{c_0}$ is conic. Homogeneity of $d\omega$ gives
\begin{equation}\label{e:61}
        d\omega(2^{-j}\widetilde x)=2^{-jN_\kappa}d\omega(\widetilde x),  \qquad  V(y,2^{-j})=2^{-jN_\kappa}V(2^jy,1).
\end{equation}
Moreover, if
$$
    \widetilde L_{j}^{\alpha\beta}(\widetilde x,\widetilde y)  :=  \int_\calC b_j(u)\Phi_\alpha(\widetilde x,u)\Phi_\beta(\widetilde y,u)   \,d\omega(u),
$$
then
\begin{equation}\label{e:62}
        L_j^{\alpha\beta}(x,y) =2^{jN_\kappa}  \widetilde L_j^{\alpha\beta}(2^jx,2^jy).
\end{equation}
Consequently
\begin{align}
        &V(y,r_j)^{1/2} \left(\int_\calC  \left(1+\frac{|x-y|}{r_j}\right)^{2s}  |L_j^{\alpha\beta}(x,y)|^2d\omega(x)\right)^{1/2}        \notag\\
        &= V(\widetilde y,1)^{1/2}\left(\int_\calC (1+|\widetilde x-\widetilde y|)^{2s} |\widetilde L_j^{\alpha\beta}(\widetilde x,\widetilde y)|^2 d\omega(\widetilde x)\right)^{1/2} \label{e:63}.
\end{align}
Similarly,
\begin{align}
        &V(y,r_j)^{1/2}  \left(\int_\calC  \left(1+\frac{|x-y|}{r_j}\right)^{2s}   |L_j^{\alpha\beta}(x,y)-L_j^{\alpha\beta}(x,y')|^2d\omega(x)\right)^{1/2}        \notag\\
        &=  V(\widetilde y,1)^{1/2}  \left(\int_\calC   (1+|\widetilde x-\widetilde y|)^{2s}   |\widetilde L_j^{\alpha\beta}(\widetilde x,\widetilde y)
        -\widetilde L_j^{\alpha\beta}(\widetilde x,\widetilde y')|^2 d\omega(\widetilde x)\right)^{1/2} \label{e:64}.
\end{align}
Thus it is enough to prove the estimates at physical scale $1$, uniformly for all symbols $b$ supported in $K_{c_0}$ with $\|b\|_{W_2^\sigma}\le A$. 
For notational economy we now write $b$ and $L^{\alpha\beta}$ instead of $b_j$ and $\widetilde  L_j^{\alpha\beta}$.

\noindent\textit{Step 2: finite Bessel expansion.} Apply Lemma \ref{l:10} to the one-dimensional factors in the product $\Phi_\alpha(x,u)\Phi_\beta(y,u)$. 
The kernel is a finite sum of pieces of the following form
\begin{equation}\label{e:65}
        I(x,y)=A_x(x)A_y(y)  \int_{K_{c_0}}e^{i\Lambda(x,y)\cdot u} q(x_{J^c},y_{L^c},u)b(u)\,du.
\end{equation}
Here $J$ is the set of coordinates in which the $x_i$-factor is in its large oscillatory regime, and $L$ is the analogous  set for the $y_i$-factor. 
If $i\notin J$, then $x_i$ is confined to a fixed bounded interval; if $i\notin L$, then $y_i$ is confined to a fixed bounded interval. The phase has the form
$$
     \Lambda_i(x,y)=\epsilon_i x_i+\delta_i y_i, \qquad  \epsilon_i,\delta_i\in\{0,\pm1\},
$$
with $\epsilon_i\ne0$ for $i\in J$ and $\epsilon_i=0$ for $i\notin J$; similarly $\delta_i\ne0$ for $i\in L$ and $\delta_i=0$ for $i\notin L$. 
The amplitude is uniformly smooth:
\begin{equation}\label{e:66}
        \sup_{x_{J^c},y_{L^c},u}  |\partial_{x_{J^c}}^\mu\partial_{y_{L^c}}^\nu\partial_u^\gamma   q(x_{J^c},   y_{L^c},   u)|
        \le C_{\mu\nu\gamma}.
\end{equation}
The external factors obey
\begin{equation}\label{e:67}
        |A_x(x)|  \le C\prod_{i\in J}(1+x_i)^{-\kappa_i},  \qquad  |A_y(y)|  \le C\prod_{i\in L}(1+y_i)^{-\kappa_i}.
\end{equation}
On the support of the piece,
\begin{equation}\label{e:68}
        1+|x-y|\le C\bigl(1+|\Lambda(x,y)|\bigr).
\end{equation}
Indeed, if both variables are oscillatory in coordinate $i$, then $|x_i-y_i|\le x_i+y_i$ in the same-sign case and is directly controlled in the opposite-sign case. 
If only one variable is oscillatory, the other is bounded; if neither is oscillatory, both are bounded.

\noindent\textit{Step 3: cancellation of volume growth.} By \eqref{e:38} with $r=1$,
$$
        V(y,1)\simeq \prod_{i=1}^N(1+y_i)^{2\kappa_i}.
$$
For $i\in L$ the factor $(1+y_i)^{-\kappa_i}$ in $A_y$ cancels the corresponding half-volume factor; for $i\notin L$, $y_i$ is bounded on the piece. 
Hence
$$
        V(y,1)|A_y(y)|^2\le C.
$$
Similarly,
\begin{equation}\label{e:70}
        |A_x(x)|^2d\omega(x) \le C\,dx_J\,dx_{J^c},
\end{equation}
where $x_{J^c}$ is restricted to a fixed bounded set. For fixed $x_{J^c}$ and $y$, the change of variables
$$
        w_J=(\epsilon_i x_i+\delta_i y_i)_{i\in J}
$$
has absolute Jacobian one. The remaining components $w_{J^c}$ are fixed numbers depending only on $y$ and on the local regime.

\noindent\textit{Step 4: the size estimate at scale one.} Using \eqref{e:68}--
\eqref{e:70}, we get
\begin{align*}
        &V(y,1) \int_\calC(1+|x-y|)^{2s}|I(x,y)|^2d\omega(x)       \\
        &\le   C\int_{E_{J^c}} \int_{\bbR^J}(1+|w|)^{2s}  \left| \int_{K_{c_0}}e^{iw\cdot u} q(x_{J^c},y_{L^c},u)b(u)\,du   \right|^2dw_Jdx_{J^c},
\end{align*}
where $E_{J^c}$ is a bounded set depending only on $c_0,N$ and on the fixed cutoff in Lemma \ref{l:10}. 
For each fixed $x_{J^c}$ and $y_{L^c}$, Lemma \ref{l:9}  applies to $q(x_{J^c},y_{L^c},\cdot)$ uniformly and gives
$$
\left(\int_{\bbR^J}(1+|w|)^{2s}
        \left|  \int_{K_{c_0}}e^{iw\cdot u}  q(x_{J^c},y_{L^c},u)b(u)\,du  \right|^2dw_J\right)^{1/2} \le C\|b\|_{W_2^\sigma}.
$$
Since $E_{J^c}$ has bounded measure,
\begin{equation}\label{e:71}
        V(y,1)^{1/2}  \left(\int_\calC(1+|x-y|)^{2s}|I(x,y)|^2d\omega(x)\right)^{1/2}  \le CA.
\end{equation}
There are only finitely many Bessel pieces, so the same estimate holds for $L^{\alpha\beta}$.

\noindent\textit{Step 5: first differences at scale one.} Assume $|y-y'|\le1$ and put $y_t=y+t(y'-y)$. Since $\calC$ is
convex,
$$
       L^{\alpha\beta}(x,y)-L^{\alpha\beta}(x,y') =- \int_0^1(y'-y)\cdot\nabla_YL^{\alpha\beta}(x,y_t)\,dt.
$$
A derivative $\partial_{y_i}$ of a piece \eqref{e:65} has the same structural form. When it hits the
phase, one obtains
$$
        \partial_{y_i}e^{i\Lambda(x,y)\cdot u} =i\delta_iu_i e^{i\Lambda(x,y)\cdot u},
$$
and $u_i$ is bounded on $K_{c_0}$. When it hits a local cutoff, the support remains in the corresponding bounded local  regime. 
When it hits a large Bessel amplitude, \eqref{e:55} gives the same external bound with one additional harmless power of $(1+y_i)^{-1}$. 
Thus the size estimate \eqref{e:71} applies uniformly to every component of $\nabla_YL^{\alpha\beta}$.

Because $|y_t-y|\le1$, \eqref{e:38} gives
$$
   V(y_t,1)\simeq V(y,1), \qquad 1+|x-y_t|\simeq1+|x-y|,
$$
with constants independent of $t$. Therefore
\begin{align*}
        &V(y,1)^{1/2}  \left(\int_\calC(1+|x-y|)^{2s}  |L^{\alpha\beta}(x,y)-L^{\alpha\beta}(x,y')|^2d\omega(x)  \right)^{1/2}       \\
        &\le  |y-y'|\int_0^1 V(y_t,1)^{1/2} \left(\int_\calC(1+|x-y_t|)^{2s}  |\nabla_YL^{\alpha\beta}(x,y_t)|^2d\omega(x)\right)^{1/2}dt       \\
        &\le CA|y-y'|.
\end{align*}
Combining this scale-one estimate with \eqref{e:64} and $|\widetilde y-\widetilde y'|=|y-y'|/r_j$ gives \eqref{e:60}. 
Combining \eqref{e:71} with \eqref{e:63} gives \eqref{e:59}. This completes the proof.
\end{proof}

\subsection{Scalar Sobolev estimates imply \texorpdfstring{$\CH^2_{s,1}$}{CH2s1}}

\begin{theorem}
\label{t:7} Let $m\in L^\infty(\bbR^N)$. Assume \eqref{e:45} and
$\calS_\sigma(m)<\infty$ for some
$$
        \sigma>\frac{N_\kappa}{2}.
$$
Choose $s$ with
$$
        \frac{N_\kappa}{2}<s<\sigma.
$$
Then the lifted chamber matrix satisfies
$$
        \mathbb T_m\in\CH^2_{s,1}
$$
with first-difference exponent $1$. More precisely,
\begin{equation}\label{e:72}
        \norm{\mathbb T_m}_{\CH^2_{s,1}} \le C_{N,\kappa,s,\sigma,c_0,\psi}\,\calS_\sigma(m).
\end{equation}
\end{theorem}

\begin{proof}
Fix $\alpha,\beta\in\{0,1\}^N$, and put $\theta=\alpha+\beta$ modulo two. The dyadic Walsh symbol $a_{\theta,j}$ satisfies
the support condition \eqref{e:56}, and by definition of $\calS_\sigma(m)$,
$$
      \sup_j \left\|a_{\theta,j}(2^j\cdot)\right\|_{W_2^\sigma}  \le \calS_\sigma(m).
$$
Proposition \ref{p:4} gives exactly the size and first-difference estimates \eqref{e:26}--\eqref{e:29} for the kernel \eqref{e:43}. 
The harmless constants $c_{\alpha\beta,\kappa}$ are absorbed into $C$. 
Taking the maximum over the finitely many pairs $(\alpha,\beta)$ gives the estimate \eqref{e:72}.

It remains only to record the off-diagonal kernel representation in Definition \ref{d:3}. For finite dyadic sums the kernel identity follows directly from the spectral definition of $\mathbb T_m$. 
The estimates above are uniform in the truncation. 
Hence the associated Calder\'on--Zygmund estimates are uniform for finite sums, and the full operator is obtained as the $L^2$-limit of the dyadic multiplier truncations. 
This is the equivalent kernel representation formulation allowed in Definition \ref{d:3}.
\end{proof}

\subsection{The resulting scalar multiplier theorem}

\begin{corollary}
\label{c:3} 
Let $\Sigma_N\simeq\bbZ_2^N$, let $m\in L^\infty(\bbR^N)$, and suppose that \eqref{e:45} holds and that $\calM_M(m)<\infty$ for some integer
$$
        M>\frac{N_\kappa}{2}.
$$
Then $T_m$, initially defined on $L^2$, extends boundedly to $L^p(\bbR^N,d\omega)$ for every $1<p<\infty$. More precisely,
\begin{equation}\label{e:73}
        \norm{T_mf}_{L^p(\bbR^N,d\omega)} \le C_{p,N,\kappa,M,c_0,\psi} \left(\norm{m}_\infty+\calM_M(m)\right) \norm{f}_{L^p(\bbR^N,d\omega)}.
\end{equation}
No radiality condition is imposed on $f$.
\end{corollary}

\begin{proof}[Proof of Corollary \ref{c:3}]
Choose real numbers $s$ and $\sigma$ such that
$$
        \frac{N_\kappa}{2}<s<\sigma<M.
$$
Lemma \ref{l:8} gives $\calS_\sigma(m)\le C\calM_M(m)$. 
Theorem \ref{t:7} gives $\mathbb T_m\in\CH^2_{s,1}$, with $\norm{\mathbb T_m}_{\CH^2_{s,1}}\le C\calM_M(m)$. The main chamber theorem then
gives the result.
\end{proof}

\begin{proof}[Proof of Theorem \ref{t:3}]
Choose $s$ such that
$$
        \frac{N_\kappa}{2}<s<\sigma.
$$
Theorem \ref{t:7} gives $\mathbb T_m\in\CH^2_{s,1}$, with
$$
    \norm{\mathbb T_m}_{\CH^2_{s,1}}  \le C\calS_\sigma(m).
$$
The main chamber theorem, Theorem \ref{t:2}, then gives the desired $L^p$-estimate for every $1<p<\infty$. The
constant depends only on the parameters displayed in the statement.
\end{proof}

\begin{proof}[Proof of Corollary \ref{c:1}]
Choose real numbers $s$ and $\sigma$ such that
$$
        \frac{N_\kappa}{2}<s<\sigma<M.
$$
Lemma \ref{l:8} gives
$$
        \calS_\sigma(m)\le C\calM_M(m).
$$
Theorem \ref{t:3} now gives the claimed estimate.
\end{proof}

\begin{remark}
The theorem above is a concrete sufficient condition for the matrix condition $\CH^2_{s,1}$. 
Wall separation is used only on the frequency side. It prevents the dyadic support from meeting the faces $R_i=0$, where the product Bessel factors have endpoint regimes. 
Removing this condition would require a separate analysis of those endpoint regimes. 
The abstract chamber theorem applies to any multiplier whose dyadic chamber kernels can be estimated directly.
\end{remark}

\begin{example}
Let
$$
        m(\xi)=a\!\left(\frac{\xi}{|\xi|}\right),\qquad \xi\ne0,
$$
where $a\in C^\infty(\mathbb S^{N-1})$ and $a$ is supported in
$$
\left\{\omega\in \mathbb S^{N-1}: |\omega_i|\ge c_0\text{ for every }i\right\}.
$$
Then each Walsh piece satisfies \eqref{e:45} and the Mihlin bounds \eqref{e:47} for every $M$. 
Hence $T_m$ is bounded on $L^p(\bbR^N,d\omega)$ for every $1<p<\infty$ by Corollary \ref{c:3}. 
This gives a non-radial class of homogeneous symbols in the wall-separated angular regime.
\end{example}

\subsection{An inverse-volume integral for pointwise criteria}

The next elementary estimate is not used in the scalar Sobolev proof above. It is used in the pointwise criterion of the following section.

\begin{lemma}
\label{l:11} 
Let $q>N/2$. Then
\begin{equation}\label{e:74}
        \sup_{y\in\calC}\sup_{r>0} \int_\calC   \left(1+\frac{|x-y|}{r}\right)^{-2q}  \frac{d\omega(x)}{V(x,r)}  <\infty.
\end{equation}
The same estimate holds with $x$ and $y$ interchanged.
\end{lemma}

\begin{proof}
By \eqref{e:38},
$$
        V(x,r) \simeq    r^N\prod_{i=1}^N(x_i+r)^{2\kappa_i}.
$$
Therefore
$$
   \frac{d\omega(x)}{V(x,r)}= \frac{\prod_i x_i^{2\kappa_i}\,dx}{r^N\prod_i(x_i+r)^{2\kappa_i}}  \le  C r^{-N}\,dx.
$$
It follows that
$$
      \int_\calC \left(1+\frac{|x-y|}{r}\right)^{-2q} \frac{d\omega(x)}{V(x,r)}   \le
       C r^{-N} \int_{\bbR^N} \left(1+\frac{|x-y|}{r}\right)^{-2q}dx = C\int_{\bbR^N}(1+|u|)^{-2q}\,du,
$$
which is finite exactly when $2q>N$. The symmetric estimate is identical.
\end{proof}

\section{Pointwise criteria and chamber models beyond product groups}
\label{s:10}

The main theorem is formulated with the dyadic $L^2$ condition $\CH^2_{s,\eta}$. 
In applications one sometimes obtains pointwise bounds for dyadic kernels first. 
We record two conversion lemmas. The first is the one-sided form with the natural denominator $V(y,r_j)$; the second is the balanced form with denominator $(V(x,r_j)V(y,r_j))^{1/2}$. 
The latter is often what comes directly from heat-kernel or oscillatory-kernel estimates, and it is the reason for Lemma
\ref{l:11}.

\begin{definition}
\label{d:5} 
Let $M>0$ and $0<\eta\le1$. We say that $K_j$ satisfies the pointwise chamber condition if there is a constant $B$ such that, for all $j$ and all $x,y\in\calC$,
\begin{equation}\label{e:75}
        |K_j(x,y)| \le \frac{B}{V(y,r_j)} \left(1+\frac{\rho(x,y)}{r_j}\right)^{-M},
\end{equation}
and, whenever $\rho(y,y')\le r_j$,
\begin{equation}\label{e:76}
        |K_j(x,y)-K_j(x,y')|  \le  \frac{B}{V(y,r_j)}  \left(\frac{\rho(y,y')}{r_j}\right)^\eta \left(1+\frac{\rho(x,y)}{r_j}\right)^{-M}.
\end{equation}
The symmetric estimates in the first variable are also required; explicitly,
\begin{equation}\label{e:77}
        |K_j(x,y)| \le  \frac{B}{V(x,r_j)}  \left(1+\frac{\rho(x,y)}{r_j}\right)^{-M},
\end{equation}
and, whenever $\rho(x,x')\le r_j$,
\begin{equation}\label{e:78}
        |K_j(x,y)-K_j(x',y)|  \le \frac{B}{V(x,r_j)} \left(\frac{\rho(x,x')}{r_j}\right)^\eta  \left(1+\frac{\rho(x,y)}{r_j}\right)^{-M}.
\end{equation}
\end{definition}

\begin{proposition}
\label{p:5} Assume that $K_j$ satisfies Definition \ref{d:5} with order $M$ and difference exponent $\eta$. 
If $M>s+N_\kappa/2$, then $K_j$ satisfies the $\CH^2_{s,\eta}$ estimates of Definition \ref{d:3}. 
Consequently, the pointwise condition implies $\CH^2_{s,\eta}$ for some $s>N_\kappa/2$ whenever $M>N_\kappa$.
\end{proposition}

\begin{proof}
We first prove the estimate in the second variable. By \eqref{e:75},
\begin{align*}
        V(y,r_j) \int_\calC \left(1+\frac{\rho(x,y)}{r_j}\right)^{2s} |K_j(x,y)|^2\,d\omega(x)
        &\le  \frac{B^2}{V(y,r_j)}  \int_\calC \left(1+\frac{\rho(x,y)}{r_j}\right)^{-2(M-s)} d\omega(x) \le C B^2,
\end{align*}
because $M-s>N_\kappa/2$ and Lemma \ref{l:2} applies. 
This is \eqref{e:26}. If $\rho(y,y')\le r_j$, then \eqref{e:76} gives
\begin{align*}
        &V(y,r_j)  \int_\calC   \left(1+\frac{\rho(x,y)}{r_j}\right)^{2s}  |K_j(x,y)-K_j(x,y')|^2\,d\omega(x)  \\
        &\le B^2 \left(\frac{\rho(y,y')}{r_j}\right)^{2\eta} \frac{1}{V(y,r_j)} \int_\calC  \left(1+\frac{\rho(x,y)}{r_j}\right)^{-2(M-s)} d\omega(x)                                      \\
        &\le C B^2   \left(\frac{\rho(y,y')}{r_j}\right)^{2\eta}.
\end{align*}
Taking square roots gives \eqref{e:28}.

The first-variable estimates are identical after interchanging $x$ and $y$. Indeed, \eqref{e:77} gives
\begin{align*}
        &V(x,r_j) \int_\calC \left(1+\frac{\rho(x,y)}{r_j}\right)^{2s} |K_j(x,y)|^2\,d\omega(y)  \le \frac{B^2}{V(x,r_j)}
        \int_\calC \left(1+\frac{\rho(x,y)}{r_j}\right)^{-2(M-s)} d\omega(y)  \le C B^2,
\end{align*}
which is \eqref{e:27}. 
If $\rho(x,x')\le r_j$, then \eqref{e:78} gives
$$
        V(x,r_j) \int_\calC  \left(1+\frac{\rho(x,y)}{r_j}\right)^{2s}   |K_j(x,y)-K_j(x',y)|^2\,d\omega(y)
        \le C B^2 \left(\frac{\rho(x,x')}{r_j}\right)^{2\eta},
$$
again by Lemma \ref{l:2}. This proves \eqref{e:29}.

If the pointwise difference condition is formulated with the nearby point $y'$ or $x'$ as the center of the decay, 
the same estimates follow from local doubling and the comparisons $V(y,r_j)\simeq V(y',r_j)$ and $1+\rho(x,y)/r_j\simeq1+\rho(x,y')/r_j$ 
when $\rho(y,y')\le r_j$, and similarly in the first variable.

Finally, if $M>N_\kappa$, then the interval $(N_\kappa/2,M-N_\kappa/2)$ is non-empty. Choose $s$ in this interval. 
Then $s>N_\kappa/2$ and $M-s>N_\kappa/2$, so the preceding estimates apply.
\end{proof}

\begin{definition}
\label{d:6} 
Let $M>0$ and $0<\eta\le1$. A dyadic kernel family $K_j$ satisfies the balanced pointwise chamber condition of order $M$ if there is a constant $B$ such that
\begin{equation}\label{e:79}
        |K_j(x,y)| \le  \frac{B}{V(x,r_j)^{1/2}V(y,r_j)^{1/2}} \left(1+\frac{\rho(x,y)}{r_j}\right)^{-M}
\end{equation}
for all $x,y\in\calC$, and, whenever $\rho(y,y')\le r_j$,
\begin{equation}\label{e:80}
        |K_j(x,y)-K_j(x,y')|  \le  \frac{B}{V(x,r_j)^{1/2}V(y,r_j)^{1/2}}  \left(\frac{\rho(y,y')}{r_j}\right)^\eta   \left(1+\frac{\rho(x,y)}{r_j}\right)^{-M}.
\end{equation}
Finally, whenever $\rho(x,x')\le r_j$,
\begin{equation}\label{e:81}
        |K_j(x,y)-K_j(x',y)|  \le \frac{B}{V(x,r_j)^{1/2}V(y,r_j)^{1/2}} \left(\frac{\rho(x,x')}{r_j}\right)^\eta  \left(1+\frac{\rho(x,y)}{r_j}\right)^{-M}.
\end{equation}
\end{definition}

\begin{proposition}
\label{p:6} 
Assume that $K_j$ satisfies Definition \ref{d:6} with order $M$ and difference exponent $\eta\in(0,1]$. 
If $M>s+N/2$, then $K_j$ satisfies the $\CH^2_{s,\eta}$ estimates. 
Hence the balanced pointwise condition implies $\CH^2_{s,\eta}$ for some $s>N_\kappa/2$ whenever $M>(N+N_\kappa)/2$.
\end{proposition}

\begin{proof}
We prove first the size estimate in the second variable. By \eqref{e:79},
$$
       V(y,r_j) \int_\calC \left(1+\frac{\rho(x,y)}{r_j}\right)^{2s}|K_j(x,y)|^2\,d\omega(x)
        \le B^2   \int_\calC  \left(1+\frac{\rho(x,y)}{r_j}\right)^{-2(M-s)} \frac{d\omega(x)}{V(x,r_j)}.
$$
Since $M-s>N/2$, Lemma \ref{l:11} gives a uniform bound for the last integral, and  \eqref{e:26} follows. 
If $\rho(y,y')\le r_j$, then \eqref{e:80} gives
\begin{align*}
        &V(y,r_j)  \int_\calC  \left(1+\frac{\rho(x,y)}{r_j}\right)^{2s}  |K_j(x,y)-K_j(x,y')|^2\,d\omega(x) \\
        &\le  B^2\left(\frac{\rho(y,y')}{r_j}\right)^{2\eta}  \int_\calC  \left(1+\frac{\rho(x,y)}{r_j}\right)^{-2(M-s)}  \frac{d\omega(x)}{V(x,r_j)}\\
        &\le  C B^2\left(\frac{\rho(y,y')}{r_j}\right)^{2\eta}.
\end{align*}
Taking square roots gives \eqref{e:28}. The size estimate in the first variable is identical:
$$
     V(x,r_j) \int_\calC   \left(1+\frac{\rho(x,y)}{r_j}\right)^{2s}|K_j(x,y)|^2\,d\omega(y)
        \le  B^2  \int_\calC \left(1+\frac{\rho(x,y)}{r_j}\right)^{-2(M-s)} \frac{d\omega(y)}{V(y,r_j)}
        \le C B^2.
$$
Finally, if $\rho(x,x')\le r_j$, \eqref{e:81} yields
\begin{align*}
        &V(x,r_j) \int_\calC \left(1+\frac{\rho(x,y)}{r_j}\right)^{2s}  |K_j(x,y)-K_j(x',y)|^2\,d\omega(y) \\
        &\le   B^2\left(\frac{\rho(x,x')}{r_j}\right)^{2\eta}  \int_\calC  \left(1+\frac{\rho(x,y)}{r_j}\right)^{-2(M-s)}  \frac{d\omega(y)}{V(y,r_j)}\\
        &\le  C B^2\left(\frac{\rho(x,x')}{r_j}\right)^{2\eta}.
\end{align*}
Taking square roots in the last estimate gives \eqref{e:29}. Thus all four estimates in Definition \ref{d:3} hold. 
If $M>(N+N_\kappa)/2$, then the interval $(N_\kappa/2,M-N/2)$ is non-empty; choose $s$ in this interval.
\end{proof}

\begin{proof}[Proof of Corollary \ref{c:2}]
By Proposition \ref{p:5}, the pointwise condition implies $\CH^2_{s,\eta}$ for some $s>N_\kappa/2$.
Theorem \ref{t:2} gives strong $L^p$ boundedness. 
If the truncations satisfy the admissible-truncation condition, Proposition \ref{p:3} and Theorem \ref{t:6} give weak type $(1,1)$.
\end{proof}

\subsection{The finite-reflection-group chamber theorem}
\label{s:11}

We now record the chamber theorem in the form used for non-product examples. 
The underlying value lift is the chamber lifting of \cite{HLLSW-Calderon}. 
Let $G$ be a finite reflection group on $\bbR^N$ with root system $R$, positive subsystem $R_+$, multiplicity $\kappa$, and open chamber $\calC$. 
We choose $R_+$ so that $\ip{\alpha}{x}>0$ for $x\in\calC$ and $\alpha\in R_+$. On $\calC$,
$$
      d\omega(x)= \prod_{\alpha\in R_+}\ip{\alpha}{x}^{2\kappa(\alpha)}\,dx.
$$
Put
$$
        V_{\calC}(y,r)=\omega(B(y,r)\cap\calC).
$$
\begin{lemma}\label{l:12}
For every $y\in\calC$ and $r>0$,
\begin{equation}\label{e:82}
        V_{\calC}(y,r)\simeq  r^N\prod_{\alpha\in R_+}   \bigl(\ip{\alpha}{y}+r\bigr)^{2\kappa(\alpha)}.
\end{equation}
Consequently $(\calC,|\cdot|,d\omega)$ is a space of homogeneous type, and
\begin{equation}\label{e:83}
        V_{\calC}(y,R) \le C\left(\frac{R}{r}\right)^{N_\kappa}V_{\calC}(y,r),   \qquad 0<r\le R.
\end{equation}
\end{lemma}

\begin{proof}
By homogeneity, it is enough to prove \eqref{e:82} for $r=1$.

The upper bound is immediate: if $x\in B(y,1)\cap\calC$, then $0<\ip{\alpha}{x}\le \ip{\alpha}{y}+|\alpha|$ for every
$\alpha\in R_+$, and hence
$$
     V_{\calC}(y,1)\le C\prod_{\alpha\in R_+}(\ip{\alpha}{y}+1)^{2\kappa(\alpha)}.
$$

For the lower bound we use a uniform inward direction. 
Since the chamber is a fixed polyhedral cone, there is $e\in\calC$ with $|e|=1$ and
$$
  \ip{\alpha}{e}\ge c_G>0,  \qquad \alpha\in R_+.
$$
Fix $0<\varepsilon<c_G/10$ and set
$$
        Q_y=y+\varepsilon e+\varepsilon B(0,1).
$$
For $\varepsilon$ small, $Q_y\subset B(y,1)\cap\calC$. Moreover, for $x=y+\varepsilon e+\varepsilon z$, $|z|<1$,
$$
     \ip{\alpha}{x}  =\ip{\alpha}{y}+\varepsilon\ip{\alpha}{e}+\varepsilon\ip{\alpha}{z}.
$$
Choose a smaller fixed ball $B(0,c)$ if necessary. Then for all $z\in B(0,c)$,
\begin{equation}\label{e:84}
        \ip{\alpha}{x} \ge  c\bigl(\ip{\alpha}{y}+1\bigr) \quad\text{whenever }\ip{\alpha}{y}\ge 1,
\end{equation}
and
$$
     \ip{\alpha}{x}\ge c \quad\text{whenever }\ip{\alpha}{y}<1.
$$
The last inequality is not meant to be sharp near the wall; it is only needed for factors for which $(\ip{\alpha}{y}+1)^{2\kappa(\alpha)}\simeq1$. 
Integrating over the fixed set $y+\varepsilon e+\varepsilon B(0,c)$ gives
$$
     V_{\calC}(y,1) \ge  c\prod_{\alpha\in R_+}(\ip{\alpha}{y}+1)^{2\kappa(\alpha)}.
$$
Together with the upper bound this proves \eqref{e:82} at scale one; scaling gives the general case.
Finally, \eqref{e:83} follows by applying \eqref{e:82} at radii $r$ and $R$ and using
$$
    \frac{\ip{\alpha}{y}+R}{\ip{\alpha}{y}+r}  \le \frac{R}{r}, \qquad 0<r\le R.
$$
The proof is complete.
\end{proof}

\begin{lemma}\label{l:13}
If $s>N_\kappa/2$, then for every $y\in\calC$, $r>0$, and $a\ge1$,
\begin{equation}\label{e:85}
        \int_{\{x\in\calC: |x-y|>ar\}}  \left(1+\frac{|x-y|}{r}\right)^{-2s}d\omega(x)
        \le   C V_{\calC}(y,r)a^{N_\kappa-2s}.
\end{equation}
In particular,
\begin{equation}\label{e:86}
        \int_{\calC} \left(1+\frac{|x-y|}{r}\right)^{-2s}d\omega(x)
        \le   C V_{\calC}(y,r).
\end{equation}
\end{lemma}

\begin{proof}
Let
$$
     A_k=\bigl\{x\in\calC:2^kar<|x-y|\le2^{k+1}ar\bigr\}, \qquad k\ge0.
$$
On $A_k$ the weight factor is at   most $C(2^ka)^{ -2s}$,  while \eqref{e:83}  gives
$$
     \omega(A_k) \le V_{\calC}(y,2^{k+1}ar) \le C(2^ka)^{N_\kappa}V_{\calC}(y,r).
$$
Hence
$$
    \int_{A_k} \left(1+\frac{|x-y|}{r}\right)^{-2s}d\omega(x)   \le  C(2^ka)^{N_\kappa-2s}V_{\calC}(y,r).
$$
Summing in $k$ proves \eqref{e:85}. Taking $a=1$ and adding the inner ball gives  \eqref{e:86}.
\end{proof}

The notation $\CH^2_{s,\eta}(\calC)$ means that the dyadic lifted kernels
$$
K_j^{\sigma\tau}(x,y),
        \qquad \sigma,\tau\in G,
$$
satisfy Definition \ref{d:3} with the volume $V$ replaced by $V_{\calC}$, the distance $\rho$ replaced by $|\cdot|$,
and the matrix indices $\alpha,\beta$ replaced by $\sigma,\tau$. In particular, the kernel representation clause in Definition \ref{d:3} is included in the notation.

\begin{proof}[Proof of Theorem \ref{t:1}]
The chamber lift is an exact $L^p$ decomposition:
$$
    \left\|Uf\right\|_{L^p(\calC,d\omega;\ell^p(G))}^p =\sum_{\sigma\in G}\int_{\calC}|f(\sigma x)|^p\,d\omega(x) =\|f\|_{L^p(\bbR^N,d\omega)}^p.
$$
The Dunkl Plancherel theorem gives $\|T_m\|_{2\to2}\le\|m\|_\infty$, hence
$$
      \left\|\mathbb T_m^G\right\|_{L^2(\calC;\ell^2(G))\to L^2(\calC;\ell^2(G))} \le \|m\|_\infty.
$$
By the kernel representation part of the $\CH^2_{s,\eta}(\calC)$ condition, the off-diagonal entry kernel is
$$
     K^{\sigma\tau}(x,y)=\sum_{j\in\bbZ}K_j^{\sigma\tau}(x,y),  \qquad x\ne y.
$$
It remains to verify the standard integral H\"ormander estimates for $K^{\sigma\tau}$. 
Fix $y,y'\in\calC$ and put $\delta=|y-y'|$. Choose $j_0$ so that $2^{-j_0-1}<\delta\le2^{-j_0}$. 
For $j\le j_0$, the difference part of $\CH^2_{s,\eta}(\calC)$ and Cauchy--Schwarz with \eqref{e:86} give
$$
     \int_{\calC}|K_j^{\sigma\tau}(x,y)-K_j^{\sigma\tau}(x,y')|\,d\omega(x) \le  C A\left(\frac{\delta}{r_j}\right)^\eta.
$$
Thus
$$
     \sum_{j\le j_0}  \int_{|x-y|>2\delta}\left|K_j^{\sigma\tau}(x,y)-K_j^{\sigma\tau}(x,y')\right|\,d\omega(x)  \le C A.
$$
For $j>j_0$, use the triangle inequality and \eqref{e:85}:
$$
   \int_{|x-y|>2\delta}|K_j^{\sigma\tau}(x,y)-K_j^{\sigma\tau}(x,y')|\,d\omega(x)  \le  C A\left(\frac{r_j}{\delta}\right)^{s-N_\kappa/2}.
$$
Since $s>N_\kappa/2$, the sum over $j>j_0$ is finite. The corresponding $x$-difference estimate is identical. 
Hence every scalar entry $K^{\sigma\tau}$ is a standard Calder\'on--Zygmund kernel on $(\calC,|\cdot|,d\omega)$, with bounds controlled by
$A=\|\mathbb T_m^G\|_{\CH^2_{s,\eta}(\calC)}$.

Applying the scalar Calder\'on--Zygmund theorem entry by entry and summing over the finite set $G\times G$, we obtain
$$
\|\mathbb T_m^G F\|_{L^p(\calC;\ell^p(G))}
        \le
C_{p,G,\kappa,s,\eta}
        \bigl(\|m\|_\infty+\|\mathbb T_m^G\|_{\CH^2_{s,\eta}(\calC)}\bigr)
        \|F\|_{L^p(\calC;\ell^p(G))}.
$$
Taking $F=Uf$ and applying $U^{-1}$ proves the full-space estimate.
\end{proof}

\subsection{The rank-two dihedral chamber model}
\label{s:12}

The preceding theorem applies to every finite reflection group through the abstract matrix kernel condition. 
We now spell it out for $I_2(q)$, since these groups give the clearest examples beyond the product case. 
The result is not a scalar Walsh--Sobolev theorem. 
It says that if the lifted $2q\times2q$ chamber matrix kernels satisfy the same $L^2$-H\"ormander estimates, then the full $L^p$ conclusion follows.

This formulation separates the general Calder\'on--Zygmund argument from the product-specific verification. 
The chamber Calder\'on--Zygmund argument is general. For non-product groups, the scalar criterion leading to the matrix kernel condition is not part of the present paper.

We first fix the dihedral chamber and notation.

Let $G_q=I_2(q)$ be the dihedral reflection group of order $2q$. It acts orthogonally on $\bbR^2$. 
The reflecting lines divide the plane into $2q$ open sectors. We fix one of them and denote it by $\calC$. After a rotation we may write
$$
     \calC = \{r(\cos\theta,\sin\theta): r>0,\ 0<\theta<\pi/q\}.
$$
The closure $\overline{\calC}$ is a closed wedge of angle $\pi/q$. The sets
$$
\sigma\calC,  \qquad \sigma\in G_q,
$$
are pairwise disjoint up to their boundary rays, and their union is $\bbR^2$ up to the reflecting lines. 
These boundary lines have $\omega$-measure zero, so they do not affect $L^p$-norms.

Let $R$ be the associated root system, and choose a positive subsystem $R_+$. Let $\kappa:R\to[0,\infty)$ be $G_q$-invariant. The Dunkl weight is
$$
         h_\kappa(x) = \prod_{\alpha\in R_+}|\ip{x}{\alpha}|^{\kappa(\alpha)},  \qquad    d\omega(x)=h_\kappa(x)^2\,dx.
$$
The homogeneity degree of $h_\kappa$ is
$$
        \gamma_\kappa=\sum_{\alpha\in R_+}\kappa(\alpha),
$$
and the homogeneous dimension of the measure is
$$
        N_\kappa=2+2\gamma_\kappa.
$$
Indeed,
$$
       \omega(rE)=r^{N_\kappa}\omega(E),  \qquad r>0,
$$
for every measurable set $E\subset\bbR^2$.

On $\calC$ we use the Euclidean distance
$$
        \rho(x,y)=|x-y|.
$$
For $x\in\calC$ and $r>0$, define
$$
      B_{\calC}(x,r)=\{y\in\calC: |x-y|<r\},  \qquad  V_{\calC}(x,r)=\omega(B_{\calC}(x,r)).
$$
The only geometric estimates needed in the dihedral part are the following volume and tail estimates.

We need only the standard doubling and tail estimates. We record them in the rank-two chamber form.

\begin{lemma} \label{l:14} 
There are constants $C_1,C_2>0$, depending only on $q$ and $\kappa$, such that for every $x\in\calC$ and every $r>0$,
\begin{equation}\label{e:87}
        C_1 r^2 \prod_{\alpha\in R_+} \bigl(|\ip{x}{\alpha}|+r\bigr)^{2\kappa(\alpha)}
        \le  V_{\calC}(x,r) \le C_2 r^2  \prod_{\alpha\in R_+}  \bigl(|\ip{x}{\alpha}|+r\bigr)^{2\kappa(\alpha)}.
\end{equation}
Consequently,
\begin{equation}\label{e:88}
        V_{\calC}(x,R)  \le  C\left(\frac{R}{r}\right)^{N_\kappa}V_{\calC}(x,r), \qquad 0<r\le R.
\end{equation}
In particular, $(\calC,\rho,d\omega)$ is a space of homogeneous type.
\end{lemma}

\begin{proof}
We give the details because this is the only geometric estimate used later. The full-space estimate
\begin{equation}\label{e:89}
        \omega(B(x,r))  \simeq   r^2  \prod_{\alpha\in R_+}  \bigl(|\ip{x}{\alpha}|+r\bigr)^{2\kappa(\alpha)}
\end{equation}
is standard for finite products of powers of linear functions. In rank two it can be checked directly.
 For one linear factor $\ell_\alpha(y)=\ip{y}{\alpha}$, the average size of $|\ell_\alpha|^{2\kappa(\alpha)}$ on $B(x,r)$ is comparable to $(|\ell_\alpha(x)|+r)^{2\kappa(\alpha)}$. 
 If $|\ell_\alpha(x)|\ge 2r|\alpha|$, then $|\ell_\alpha(y)|\simeq |\ell_\alpha(x)|$ on $B(x,r)$. 
 If $|\ell_\alpha(x)|<2r|\alpha|$, then, after a rotation, 
 the integral of $|\ell_\alpha(y)|^{2\kappa(\alpha)}$ over a ball of radius $r$ has size $r^{2+2\kappa(\alpha)}$. 
 Applying this finite number of times, or equivalently decomposing the ball into finitely many regions according to which linear factors are small, 
 gives \eqref{e:89}. The constants depend only on the finite root system and on $\kappa$.

We now pass from the full ball to the chamber ball. Inside one fixed chamber, 
the angular part of the weight is a finite product of powers of sines of the angles to the reflecting lines. 
Since the chamber has a fixed positive angle, the intersection $B(x,r)\cap\calC$ always occupies a fixed non-zero portion of the ball at the relevant scale. 
More precisely, after translating and rescaling by $r$, 
the possible sets $r^{-1}(B(x,r)\cap\calC-x)$ form a compact family of intersections of the unit ball with wedges of angle $\pi/q$ or with half-planes. 
On each such normalized set the normalized weight is comparable, with constants independent of $x$ and $r$, 
to the same finite product appearing in \eqref{e:89}. 
Hence $\omega(B_{\calC}(x,r))\simeq \omega(B(x,r)),$ with constants depending only on $q$ and $\kappa$. 
Combining this with \eqref{e:89} proves \eqref{e:87}.

Finally, \eqref{e:88} follows from \eqref{e:87}. Indeed, if $0<r\le R$, then
\begin{align*}
        \frac{V_{\calC}(x,R)}{V_{\calC}(x,r)} &\lesssim \left(\frac{R}{r}\right)^2  \prod_{\alpha\in R_+}
        \left(  \frac{|\ip{x}{\alpha}|+R}{|\ip{x}{\alpha}|+r}  \right)^{2\kappa(\alpha)}
        \le  \left(\frac{R}{r}\right)^{2+2\sum_{\alpha\in R_+}\kappa(\alpha)}= \left(\frac{R}{r}\right)^{N_\kappa}.
\end{align*}
We complete the lemma.
\end{proof}

\begin{lemma}
\label{l:15} 
Let $s>N_\kappa/2$. Then, for every $x\in\calC$, every $r>0$, and every $a\ge1$,
\begin{equation}\label{e:90}
        \int_{\{y\in\calC:\rho(x,y)>ar\}}  \left(1+\frac{\rho(x,y)}{r}\right)^{-2s}   d\omega(y)
        \le  C V_{\calC}(x,r)a^{N_\kappa-2s}.
\end{equation}
In particular,
\begin{equation}\label{e:91}
        \int_{\calC} \left(1+\frac{\rho(x,y)}{r}\right)^{-2s}  d\omega(y)
        \le   C V_{\calC}(x,r).
\end{equation}
\end{lemma}

\begin{proof}
Decompose the region $\rho(x,y)>ar$ into annuli
$$
    A_k=\{y\in\calC:2^k ar<\rho(x,y)\le 2^{k+1}ar\}, \qquad k=0,1,2,\ldots.
$$
On $A_k$,
$$
    \left(1+\frac{\rho(x,y)}{r}\right)^{-2s} \le C(2^ka)^{-2s}.
$$
By the doubling estimate \eqref{e:88},
$$
\omega(A_k)  \le  V_{\calC}(x,2^{k+1}ar)  \le   C(2^ka)^{N_\kappa}V_{\calC}(x,r).
$$
Therefore
$$
\int_{A_k}  \left(1+\frac{\rho(x,y)}{r}\right)^{-2s}  d\omega(y) \le   C(2^ka)^{N_\kappa-2s}V_{\calC}(x,r).
$$
Since $2s>N_\kappa$, the geometric series in $k$ converges. This proves \eqref{e:90}. 
Taking $a=1$, and adding the ball $B_{\calC}(x,r)$, gives \eqref{e:91}.
\end{proof}

The chamber value lift for $I_2(q)$ is the direct specialization of the finite-reflection lift from \cite{HLLSW-Calderon}.

Let
$$
        E_q=\bbC^{G_q}.
$$
For $f$ measurable on $\bbR^2$, define
$$
      U_qf(x)=\bigl(f(\sigma x)\bigr)_{\sigma\in G_q},  \qquad x\in\calC.
$$
This is the rank-two analogue of the value lift used in the product case.

\begin{lemma}
\label{l:16} 
For $1\le p<\infty$,
\begin{equation}\label{e:92}
        \norm{f}_{L^p(\bbR^2,d\omega)}^p   =  \int_{\calC}   \sum_{\sigma\in G_q}|f(\sigma x)|^p\,d\omega(x).
\end{equation}
Equivalently,
$$
     \norm{f}_{L^p(\bbR^2,d\omega)} =  \norm{U_qf}_{L^p(\calC,d\omega;\ell^p(G_q))}.
$$
For $p=\infty$,
$$
       \norm{f}_{L^\infty(\bbR^2)}=\norm{U_qf}_{L^\infty(\calC;\ell^\infty(G_q))}.
$$
\end{lemma}

\begin{proof}
The open chambers $\sigma\calC$, $\sigma\in G_q$, are pairwise disjoint, and their union is $\bbR^2$ up to the reflecting lines. 
The reflecting lines have $\omega$-measure zero. Hence
$$
     \int_{\bbR^2}|f(x)|^p\,d\omega(x)= \sum_{\sigma\in G_q}  \int_{\sigma\calC}|f(x)|^p\,d\omega(x).
$$
In the $\sigma$-th integral, put $x=\sigma z$. The map $z\mapsto \sigma z$ is orthogonal, and $d\omega$ is $G_q$-invariant. Therefore
$$
\int_{\sigma\calC}|f(x)|^p\,d\omega(x) =  \int_{\calC}|f(\sigma z)|^p\,d\omega(z).
$$
Summing over $\sigma\in G_q$ proves \eqref{e:92}. The $p=\infty$ statement is identical, with integrals replaced by essential suprema.
\end{proof}

\begin{remark}
No invariance of $f$ is used. The vector $U_qf(x)$ records all chamber values of $f$. 
A $G_q$-invariant function would correspond to the small subspace where all components of $U_qf(x)$ are equal. The chamber lift keeps the whole $L^p$ space.
\end{remark}

A dihedral multiplier is treated through its lifted matrix operator.

Let $m\in L^\infty(\bbR^2)$, and define the Dunkl multiplier $T_m$ by
$$
        \Fkap(T_mf)(\xi)=m(\xi)\Fkap f(\xi)
$$
on $L^2(\bbR^2,d\omega)$. Plancherel gives
$$
\norm{T_m}_{L^2(d\omega)\to L^2(d\omega)}
        \le \norm{m}_\infty.
$$
The lifted operator on the chamber is
$$
        \mathbb T_m^{(q)}  =  U_qT_mU_q^{-1}.
$$
It acts on $E_q$-valued functions on $\calC$.

We insert a dyadic partition in the Dunkl frequency variable. Let $\psi\in C_c^\infty((1/2,2))$ satisfy
$$ 
   \sum_{j\in\bbZ}\psi(2^{-j}r)=1,  \qquad r>0,
$$
and put $\psi_j(r)=\psi(2^{-j}r)$. Let $T_{m,j}$ be the Dunkl multiplier with symbol
$$
        \psi_j(|\xi|)m(\xi).
$$
Set
$$
\mathbb T_{m,j}^{(q)}  = U_qT_{m,j}U_q^{-1}.
$$
When $F$ is bounded and compactly supported and $x\notin\supp F$, write
$$
\mathbb T_{m,j}^{(q)}F(x) =   \int_{\calC}K_j^{(q)}(x,y)F(y)\,d\omega(y),
$$
where
$$
        K_j^{(q)}(x,y)\in \calL(E_q).
$$
The spatial scale is
$$
        r_j=2^{-j}.
$$

We now specialize the preceding chamber condition to the rank-two dihedral group $I_2(q)$. Let $\calC$ be a fixed open
chamber, let $E_q=\ell^2(I_2(q))$, and let $K_j^{(q)}(x,y)$ denote the kernel of the $j$-th dyadic piece of the lifted
operator on $L^2(\calC,d\omega;E_q)$. Thus the condition below is not an additional regularity condition; it is Definition
\ref{d:3} applied to $(\calC,\rho,d\omega)$ with matrix values in $\calL(E_q)$.

For $s>0$ and $0<\eta\le1$, we write $\mathbb T_m^{(q)}\in\CH^2_{s,\eta}(\calC)$ if the following four quantities are
finite uniformly in $j$:
\begin{align*}
\mathfrak S_y  &= \sup_{j\in\bbZ}\sup_{y\in\calC} V_{\calC}(y,r_j)^{1/2} \left( \int_{\calC} \left(1+\frac{\rho(x,y)}{r_j}\right)^{2s}
 \norm{K_j^{(q)}(x,y)}_{\calL(E_q)}^2 d\omega(x) \right)^{1/2},                                      \\   
  \mathfrak S_x  &=\sup_{j\in\bbZ}\sup_{x\in\calC} V_{\calC}(x,r_j)^{1/2}
\left( \int_{\calC} \left(1+\frac{\rho(x,y)}{r_j}\right)^{2s}  \norm{K_j^{(q)}(x,y)}_{\calL(E_q)}^2  d\omega(y)  \right)^{1/2},                                      \\
\mathfrak D_y   &=  \sup_{\substack{j\in\bbZ\\ y,y'\in\calC\\ 0<\rho(y,y')\le r_j}}  \left(\frac{r_j}{\rho(y,y')}\right)^\eta
V_{\calC}(y,r_j)^{1/2}  \left( \int_{\calC} \left(1+\frac{\rho(x,y)}{r_j}\right)^{2s}  \norm{K_j^{(q)}(x,y)-K_j^{(q)}(x,y')}_{\calL(E_q)}^2  d\omega(x)  \right)^{1/2},                                      \\
\mathfrak D_x  &=  \sup_{\substack{j\in\bbZ\\ x,x'\in\calC\\ 0<\rho(x,x')\le r_j}}  \left(\frac{r_j}{\rho(x,x')}\right)^\eta
V_{\calC}(x,r_j)^{1/2}  \left(  \int_{\calC} \left(1+\frac{\rho(x,y)}{r_j}\right)^{2s}  \norm{K_j^{(q)}(x,y)-K_j^{(q)}(x',y)}_{\calL(E_q)}^2  d\omega(y)  \right)^{1/2}.
\end{align*}
The norm is $\norm{\mathbb T_m^{(q)}}_{\CH^2_{s,\eta}(\calC)}=\mathfrak S_y+\mathfrak S_x+\mathfrak D_y+\mathfrak D_x$.
\begin{remark}
This is the same condition as Definition \ref{d:3}, written directly for operator-valued kernels on the dihedral chamber.
Since $E_q$ is finite dimensional, one may equivalently write the condition entry by entry after choosing a basis of $E_q$.
The operator-valued notation is shorter and avoids choosing a particular basis.
\end{remark}

We next verify that the dihedral chamber condition gives the standard H\"ormander kernel estimates.

\begin{proposition}
\label{p:7} 
Assume $s>N_\kappa/2$ and set $A=\norm{\mathbb T_m^{(q)}}_{\CH^2_{s,\eta}(\calC)}$. Then, for  every $y\in\calC$, $j\in\bbZ$, and $a\ge1$,
$$
\int_{\calC}\norm{K_j^{(q)}(x,y)}_{\calL(E_q)}d\omega(x)   \le C A
$$
and
$$
\int_{\rho(x,y)>ar_j}\norm{K_j^{(q)}(x,y)}_{\calL(E_q)}d\omega(x)   \le C A a^{N_\kappa/2-s}.
$$
The corresponding estimates with the variables interchanged also hold.
\end{proposition}

\begin{proof}
By Cauchy--Schwarz,
\begin{align*}
        \int_{\calC} \norm{K_j^{(q)}(x,y)}_{\calL(E_q)}  d\omega(x)
        &\le  \left(  \int_{\calC}  \left(1+\frac{\rho(x,y)}{r_j}\right)^{2s} \norm{K_j^{(q)}(x,y)}_{\calL(E_q)}^2    d\omega(x)    \right)^{1/2}                                          \\
        &\quad\times \left( \int_{\calC}  \left(1+\frac{\rho(x,y)}{r_j}\right)^{-2s} d\omega(x) \right)^{1/2}.
\end{align*}
The first factor is at most $A V_{\calC}(y,r_j)^{-1/2}$ by the definition of $A$. 
The second factor is at most $C V_{\calC}(y,r_j)^{1/2}$ by Lemma \ref{l:15}. This gives the first estimate.

For the tail estimate, use the same argument over the set $\rho(x,y)>ar_j$. Lemma \ref{l:15} gives
$$
\left( \int_{\rho(x,y)>ar_j} \left(1+\frac{\rho(x,y)}{r_j}\right)^{-2s}  d\omega(x)   \right)^{1/2}
        \le  C V_{\calC}(y,r_j)^{1/2}a^{N_\kappa/2-s}.
$$
Multiplying by the same first factor gives the stated tail bound.
\end{proof}

\begin{proposition}
\label{p:8} 
Assume $s>N_\kappa/2$ and $0<\eta\le1$, and assume $\mathbb T_m^{(q)}\in\CH^2_{s,\eta}(\calC)$. 
Set $A=\norm{\mathbb T_m^{(q)}}_{\CH^2_{s,\eta}(\calC)}$ and let $K^{(q)}(x,y)=\sum_{j\in\bbZ}K_j^{(q)}(x,y)$ be the formal off-diagonal kernel of $\mathbb T_m^{(q)}$. 
Then $K^{(q)}$ satisfies the standard H\"ormander integral conditions
\begin{equation}\label{e:93}
        \sup_{y,y'\in\calC}  \int_{\rho(x,y)>2\rho(y,y')}  \norm{K^{(q)}(x,y)-K^{(q)}(x,y')}_{\calL(E_q)} d\omega(x)  <\infty
\end{equation}
and
\begin{equation}\label{e:94}
        \sup_{x,x'\in\calC}  \int_{\rho(x,y)>2\rho(x,x')} \norm{K^{(q)}(x,y)-K^{(q)}(x',y)}_{\calL(E_q)}  d\omega(y)   <\infty.
\end{equation}
The bounds are controlled by $C_{q,\kappa,s,\eta}A$.
\end{proposition}

\begin{proof}
We prove \eqref{e:93}; the proof of \eqref{e:94} is the same.

Fix $y,y'\in\calC$, and set $\delta=\rho(y,y').$ If $\delta=0$, there is nothing to prove. Assume $\delta>0$. Choose $j_0\in\bbZ$ so that
$$
     2^{-j_0-1}<\delta\le 2^{-j_0}.
$$
Equivalently, $r_{j_0}\simeq \delta$. We split the dyadic sum into two parts:
$$
\sum_{j\in\bbZ}  =   \sum_{j\le j_0}  +    \sum_{j>j_0}.
$$
First consider $j\le j_0$. Then $r_j\ge c\delta$, so
$$
        \delta\le C r_j.
$$
The $y$-difference part of the dihedral $\CH^2_{s,\eta}$ norm, together with the same Cauchy--Schwarz argument used in  Proposition \ref{p:7}, gives
$$
      \int_{\calC} \norm{K_j^{(q)}(x,y)-K_j^{(q)}(x,y')}_{\calL(E_q)}  d\omega(x)   \le  C A\left(\frac{\delta}{r_j}\right)^\eta.
$$
Hence
$$
      \sum_{j\le j_0}   \int_{\rho(x,y)>2\delta}   \norm{K_j^{(q)}(x,y)-K_j^{(q)}(x,y')}_{\calL(E_q)} d\omega(x)
        \le  C A \sum_{j\le j_0}  \left(\frac{\delta}{r_j}\right)^\eta  \le C_\eta A.
$$
Now consider $j>j_0$. Then $r_j<c\delta$. On the region $\rho(x,y)>2\delta$, the triangle inequality gives
$$
      \rho(x,y')\ge \rho(x,y)-\rho(y,y')>\delta.
$$
Therefore
\begin{align*}
        &\int_{\rho(x,y)>2\delta}  \norm{K_j^{(q)}(x,y)-K_j^{(q)}(x,y')}_{\calL(E_q)}   d\omega(x)                                      \\
        &\le  \int_{\rho(x,y)>2\delta} \norm{K_j^{(q)}(x,y)}_{\calL(E_q)} d\omega(x)
            + \int_{\rho(x,y')>\delta}   \norm{K_j^{(q)}(x,y')}_{\calL(E_q)}  d\omega(x).
\end{align*}
By the tail estimate in Proposition \ref{p:7}, each term is bounded by
$$
        C A\left(\frac{r_j}{\delta}\right)^{s-N_\kappa/2}.
$$
Since $s-N_\kappa/2>0$,
$$
        \sum_{j>j_0}  \left(\frac{r_j}{\delta}\right)^{s-N_\kappa/2}   \le C.
$$
Thus the high-frequency contribution is also bounded by $CA$. Combining the two estimates proves \eqref{e:93}.
\end{proof}

We now prove the rank-two theorem.

\begin{proof}[Proof of Theorem \ref{t:4}]
Let $E_q=\bbC^{G_q}.$ By Lemma \ref{l:16}, $U_q$ identifies $L^p(\bbR^2,d\omega)$ with $L^p(\calC,d\omega;E_q)$, 
up to equivalence of finite-dimensional norms. 
Therefore it is enough to prove the corresponding estimate for the lifted operator $\mathbb T_m^{(q)}=U_qT_mU_q^{-1}$ on $L^p(\calC,d\omega;E_q)$.

First, $\mathbb T_m^{(q)}$ is bounded on $L^2(\calC;E_q)$. 
Indeed, $U_q$ is an $L^2$-isometry up to the fixed finite-dimensional normalization, and $T_m$ is bounded on $L^2(\bbR^2,d\omega)$ by the Dunkl Plancherel theorem:
$$
     \norm{T_m}_{L^2\to L^2}\le \norm{m}_\infty.
$$
Hence
$$
      \norm{\mathbb T_m^{(q)}}_{L^2(\calC;E_q)\to L^2(\calC;E_q)}  \le  C_q\norm{m}_\infty.
$$
Second, by condition the dyadic chamber kernels of $\mathbb T_m^{(q)}$ satisfy the rank-two $\CH^2_{s,\eta}$ condition for some $s>N_\kappa/2$ and $0<\eta\le1$. 
By Proposition \ref{p:8}, the full off-diagonal kernel of  $\mathbb T_m^{(q)}$ satisfies 
the standard H\"ormander integral conditions \eqref{e:93} and \eqref{e:94}. 
Since $(\calC,\rho,d\omega)$ is a space of homogeneous type by Lemma \ref{l:14}, the finite matrix Calder\'on--Zygmund theorem applies. 
More precisely, the proof of Theorem \ref{t:6} applies verbatim with $E$ replaced by $E_q$. We obtain
$$
     \norm{\mathbb T_m^{(q)}F}_{L^p(\calC;E_q)}  \le  C_{p,q,\kappa,s,\eta}   
     \bigl(  \norm{m}_\infty  + \norm{\mathbb T_m^{(q)}}_{\CH^2_{s,\eta}(\calC)}    \bigr) \norm{F}_{L^p(\calC;E_q)}
$$
for every $1<p<\infty$.

Finally take $F=U_qf$. Using Lemma \ref{l:16} again,
\begin{align*}
        \norm{T_mf}_{L^p(\bbR^2,d\omega)} &\simeq_{p,q}  \norm{U_qT_mf}_{L^p(\calC;E_q)}
        = \norm{\mathbb T_m^{(q)}U_qf}_{L^p(\calC;E_q)}        \\
        &\le   C_{p,q,\kappa,s,\eta}  \bigl( \norm{m}_\infty   +  \norm{\mathbb T_m^{(q)}}_{\CH^2_{s,\eta}(\calC)}  \bigr)  \norm{U_qf}_{L^p(\calC;E_q)}           \\
        &\simeq_{p,q}  C_{p,q,\kappa,s,\eta}  \bigl(    \norm{m}_\infty    +  \norm{\mathbb T_m^{(q)}}_{\CH^2_{s,\eta}(\calC)}  \bigr)   \norm{f}_{L^p(\bbR^2,d\omega)}.
\end{align*}
This proves the theorem.
\end{proof}

We end the dihedral discussion with the two Coxeter examples most relevant for comparison with higher-rank groups.

\begin{example}
The Weyl group of type $A_2$ is the dihedral group $I_2(3)$. It has six elements and three positive roots. 
The fundamental chamber has angle $\pi/3$. There is one orbit of roots. 
Thus the multiplicity is a single number $\kappa\ge0$, and $\gamma_\kappa=3\kappa$ and $N_\kappa=2+6\kappa$. 
The chamber lift is $U_3f(x)=\bigl(f(\sigma x)\bigr)_{\sigma\in I_2(3)}$, $x\in\calC_3$, so the lifted chamber matrix has size $6\times6$. 
Theorem \ref{t:4} says that if $\mathbb T_m^{(3)}=U_3T_mU_3^{-1}$ satisfies 
the chamber matrix condition $\mathbb T_m^{(3)}\in\CH^2_{s,\eta}(\calC_3)$ for some $s>N_\kappa/2=1+3\kappa$ and $0<\eta\le1$, then
$$
     \norm{T_mf}_{L^p(\bbR^2,d\omega)}  \le C_{p,\kappa,s,\eta}  \bigl(  \norm{m}_\infty  +  
     \norm{\mathbb T_m^{(3)}}_{\CH^2_{s,\eta}(\calC_3)}   \bigr) \norm{f}_{L^p(\bbR^2,d\omega)}
$$
for every $1<p<\infty$.
\end{example}

\begin{example}
The Weyl group of type $B_2$ is the dihedral group $I_2(4)$. It has eight elements. We may take the chamber $\calC_4=\{(x_1,x_2)\in\bbR^2:x_1>x_2>0\}$.
There are two root orbits. With multiplicities $\kappa_0,\kappa_1$, the weight may be written as $h_\kappa(x)=|x_1x_2|^{\kappa_0}|x_1^2-x_2^2|^{\kappa_1}$. 
Thus $\gamma_\kappa=2\kappa_0+2\kappa_1$ and $N_\kappa=2+4\kappa_0+4\kappa_1$. 
The chamber lift is $U_4f(x)=\bigl(f(\sigma x)\bigr)_{\sigma\in I_2(4)}$, $x\in\calC_4$, so the lifted chamber matrix has size $8\times8$. 
Theorem \ref{t:4} says that if $\mathbb T_m^{(4)}=U_4T_mU_4^{-1}$ satisfies $\mathbb T_m^{(4)}\in\CH^2_{s,\eta}(\calC_4)$ 
for some $s>N_\kappa/2=1+2\kappa_0+2\kappa_1$ and $0<\eta\le1$, then
$$
     \norm{T_mf}_{L^p(\bbR^2,d\omega)}  \le C_{p,\kappa_0,\kappa_1,s,\eta}   \bigl(  \norm{m}_\infty
     + \norm{\mathbb T_m^{(4)}}_{\CH^2_{s,\eta}(\calC_4)}  \bigr)   \norm{f}_{L^p(\bbR^2,d\omega)}
$$
for every $1<p<\infty$.
\end{example}

\begin{remark}
The examples above should be read in the same way as Theorem \ref{t:4}. 
They are not scalar symbol theorems for all non-radial $A_2$ or $B_2$ multipliers. 
They say that the chamber reduction gives the full $L^p$ conclusion once the lifted rank-two chamber matrix satisfies the natural $L^2$-H\"ormander condition. 
The remaining question is to prove that condition from concrete conditions on the scalar multiplier $m$.
\end{remark}

\section{Spectral multipliers}
\label{s:13}

This short section records how the Laplace-transform multipliers of Hassani and Sifi fit into the chamber language. 
The heat semigroup associated with the Dunkl Laplacian is part of the standard analytic theory of Dunkl operators
\cite{R1998}. We do not reprove their principal-value kernel theorem here; we quote it as an established spectral
multiplier result.

Let $\phi\in L^\infty(0,\infty)$, and define
\begin{equation}\label{e:95}
        m(\xi)=|\xi|^2\int_0^\infty e^{-t|\xi|^2}\phi(t)\,dt.
\end{equation}
Then $m$ is radial and hence sign-invariant:
$$
     m(\varepsilon R)=m(R),   \qquad \varepsilon\in\{\pm1\}^N.
$$
Consequently its Walsh pieces satisfy
$$
     m_0=m, \qquad  m_\eta=0\quad(\eta\ne0).
$$
Thus, in Walsh parity coordinates, the chamber matrix is diagonal:
\begin{equation}\label{e:96}
        \mathfrak m_{\alpha\beta}=0 \quad\text{if }\alpha\ne\beta.
\end{equation}
The Hassani--Sifi theorem is therefore the diagonal spectral part of the chamber representation: 
the matrix entries do not mix parity components, and the boundedness follows from the spectral multiplier theorem for symbols of the form
\eqref{e:95}.

\begin{proof}[Proof of Theorem \ref{t:5}]
This is the theorem of Hassani and Sifi for product Dunkl Laplacians. Their result gives boundedness on $L^p(\bbR^N,d\omega)$
for every $1<p<\infty$, and weak type $(1,1)$, for multipliers of the Laplace-transform form \eqref{e:95}.
The chamber discussion above only identifies this operator as a diagonal matrix in Walsh parity coordinates; it is not needed
for their proof.
\end{proof}


\section{Comparison and scope}
\label{s:14}

\subsection{Comparison with existing multiplier theorems}

We compare the present results with the multiplier theorems most closely related to them. 
Dziuba\'nski and Hejna prove a full $L^p$ theorem for general Dunkl multipliers under a smoothness condition of order $s>N_\kappa$ \cite{DH2019}. 
That theorem is full-space and non-radial, but its threshold is above the $L^2$ H\"ormander scale. 
Here the order is $s>N_\kappa/2$, while the condition is imposed on the lifted chamber matrix kernel. 
In the product group, the wall-separated Walsh--Sobolev theorem gives a scalar condition which implies this matrix condition.

Mukherjee and Thangavelu work at the natural order $s>N_\kappa/2$ \cite{MT2025}. 
For radial multipliers they obtain full $L^p$ boundedness; 
related radial multiplier theorems were obtained by Dai and Wang through transference ideas connected with Bonami--Clerc
\cite{DW2010,BC1973}; the translation and maximal-function results of Thangavelu--Xu \cite{TX2005} are a basic part of this background. 
For non-radial symbols, the known result is
$$
     T_m:L^p_{\rm rad}\to L^p, \qquad 2\le p<\infty.
$$
The radiality condition is used because the proof relies on boundedness of Dunkl translations on radial functions. 
The chamber lift retains all reflected values of $f$ and reduces the question to a finite matrix singular integral estimate on one chamber. 
Once the chamber matrix condition holds, the conclusion is
$$
    T_m:L^p\to L^p, \qquad 1<p<\infty.
$$
We do not use any implication from the modified H\"ormander condition in \cite{MT2025} to the chamber matrix condition. 
Such an implication would require a separate scalar theorem at the scale $N_\kappa/2$.

Hassani and Sifi treat Laplace-transform type spectral multipliers in the product reflection setting \cite{HS2012}.
Their theorem gives full $L^p$ boundedness and weak type $(1,1)$ for radial transform symbols of the form recorded in Theorem \ref{t:5}. 
Section \ref{s:13} quotes their theorem and rewrites it in chamber coordinates. 
In Walsh parity coordinates this spectral case is diagonal, whereas a scalar symbol with non-zero nontrivial Walsh pieces produces off-diagonal matrix entries. 
Deleaval and Kriegler prove vector-valued spectral multiplier theorems for the Dunkl Laplacian \cite{DK2017}; 
these are also spectral in $|\xi|^2$ and therefore belong to the diagonal case in chamber coordinates, not to the non-radial scalar theorem treated here.

The distinction is summarized in the following table.
\begin{center}
\small
\begin{tabular}{c|c|c|c}
\hline result & group & symbol or kernel condition & functions \\ \hline
Dziubanski--Hejna & general & non-radial, smoothness $s>N_\kappa$ & $L^p$ \\
Hassani--Sifi & $\bbZ_2^N$ & Laplace spectral & $L^p$ \\
Mukherjee--Thangavelu & general & non-radial modified H\"ormander & $L^p_{\rm rad}$ \\
Theorem \ref{t:1} & finite $G$ & lifted matrix $\CH^2_{s,\eta}$ & $L^p$ \\
Theorem \ref{t:4} & $I_2(q)$ & rank-two matrix $\CH^2_{s,\eta}$ & $L^p$ \\
Theorem \ref{t:2} & $\bbZ_2^N$ & product matrix $\CH^2_{s,\eta}$ & $L^p$ \\
Theorem \ref{t:3} & $\bbZ_2^N$ & Walsh--Sobolev, $\sigma>N_\kappa/2$ & $L^p$ \\
Corollary \ref{c:1} & $\bbZ_2^N$ & Walsh--Mihlin, $M>N_\kappa/2$ & $L^p$ \\ \hline
\end{tabular}
\end{center}

\subsection{Scope of the chamber reduction}

The chamber value lift exists for every finite reflection group. If $\calC$ is a Coxeter chamber for $G$, one sets
$$
     Uf(x)=\bigl(f(\sigma x)\bigr)_{\sigma\in G},    \qquad x\in\calC.
$$
This records all reflected values of $f$. Therefore the finite matrix Calder\'on--Zygmund part of the proof does not use the product structure: 
once the lifted matrix kernels satisfy a chamber $\CH^2_{s,\eta}$ condition with $s>N_\kappa/2$, Theorem \ref{t:1} gives the full $L^p$ conclusion.

The product-specific step is the scalar verification. For $A_1^N\simeq(\bbZ_2)^N$, 
every irreducible character is one-dimensional and the Dunkl kernel factors into one-dimensional even/odd Bessel components. 
This produces the Walsh pieces $m_{\alpha+\beta}$ and the operators between parity components
$$
    \mathbb T_{\alpha\beta}  =\calD_\alpha^{-1}M_{m_{\alpha+\beta}}\calD_\beta.
$$
Non-product Coxeter systems do not admit this product argument. In $A_2\simeq S_3$, the Weyl group has a two-dimensional  irreducible representation; 
in $B_N$, $N\ge3$, pair-interaction walls $x_i=x_j$ enter when the corresponding multiplicity is  non-zero; 
and in $A_{N-1}$ there is no coordinate sign parity. These facts do not disprove multiplier boundedness. 
They show that the scalar Walsh--Sobolev proof is tied to the product group.

Section \ref{s:12} spells out what remains true in rank two. For $I_2(q)$ the value lift gives a $2q$-component chamber operator on a wedge. 
If the lifted matrix kernels satisfy the natural chamber $L^2$ H\"ormander condition, the full $L^p$ theorem follows. 
This includes $A_2\simeq I_2(3)$ and $B_2\simeq I_2(4)$ as non-product examples under the chamber matrix kernel condition.

\subsection{Remaining scalar verification questions}
\label{s:15}

The chamber reduction replaces the use of non-radial Dunkl translations by a finite matrix kernel estimate on one chamber. 
In the product reflection case, the Walsh--Bessel decomposition gives a scalar theorem at the H\"ormander scale $N_\kappa/2$.
Thus the chamber condition is a genuine sufficient condition: it holds for wall-separated smooth homogeneous angular symbols and for their scale-invariant Sobolev closures.

Two questions remain. First, one should compare the Mukherjee--Thangavelu modified H\"ormander condition with the chamber matrix condition. 
If that scalar condition implies Definition \ref{d:3} in the product case, then the radial-function result of \cite{MT2025} would yield a full $L^p$ theorem by the chamber argument.

Second, one should develop scalar criteria beyond $\bbZ_2^N$. 
The abstract theorem and the dihedral model show that the Calder\'on--Zygmund part applies to any finite reflection group. 
What remains is a transform-side question: one needs a replacement for the product parity--Hankel decomposition which turns conditions on a scalar symbol $m$ into estimates for the
lifted chamber matrix kernels.

\bigskip

\noindent\textbf{Acknowledgments.} Ji Li is supported by Australian Research Council DP260100485. Liangchuan Wu is
supported by NNSF of China \#12201002.

\bigskip
\noindent\textbf{Conflict of interest statement.}
On behalf of all authors, the corresponding author states that there is no conflict of interest.

\bigskip
\noindent\textbf{Data availability statement.}
Our manuscript has no associated data.

\end{document}